\documentclass[11pt,hyphens,dvipsnames]{article}
\usepackage[utf8]{inputenc} 
\usepackage[T1]{fontenc} 
\usepackage[margin=1.0in]{geometry}

\usepackage{xcolor}
\usepackage{hyperref} 
\hypersetup{
    colorlinks=true,
    linkcolor=blue,
    filecolor=magenta,      
    urlcolor=Magenta,
    citecolor=Cerulean,
    pdftitle={Classification of Consistent Conjectural Variations Equilibria},
    }
\usepackage[round,sort&compress]{natbib}
\let\cite\citep
\usepackage{float}
\usepackage{url}
\usepackage{graphicx} 

\usepackage{subfig}
\usepackage{caption}
\usepackage{amsmath,amssymb,amsfonts,mathtools,amsthm, thmtools}
\theoremstyle{definition}
\DeclareMathOperator{\spec}{\mathrm{spec}}

\newtheorem{assumption}{Assumption}

\newtheorem{lemma}{Lemma}

\newtheorem{corollary}{Corollary}
\newtheorem{definition}{Definition}
\newtheorem{proposition}{Proposition}
\newtheorem{remark}{Remark}

\title{
The Impact of Social Value Orientation on \\ Nash Equilibria of Two Player Quadratic Games 
}
\author{Dan Calderone, Meeko Oishi}

\begin{document}
\maketitle

\begin{abstract}
We consider two player quadratic games in a cooperative framework known as {\em social value orientation}, motivated by the need to account for complex interactions between humans and autonomous agents in dynamical systems.  Social value orientation is a framework from psychology, that posits that each player incorporates the other player's cost into their own objective function, based on an individually pre-determined degree of cooperation.  The degree of cooperation determines the weighting that a player puts on their own cost relative to the other player's cost.  We characterize the Nash equilibria of two player quadratic games under social value orientation by creating expansions that elucidate the relative difference between this new equilibria (which we term the SVO-Nash equilibria) and more typical equilibria, such as the competitive Nash equilibria, individually optimal solutions, and the fully cooperative solution.    
Specifically, each expansion parametrizes the space of cooperative Nash equilibria as a family of one-dimensional curves where each curve is computed by solving an eigenvalue problem.  We show that both bounded and unbounded equilibria may exist.  For equilibria that are bounded, we can identify bounds as the intersection of various ellipses; for equilibria that are unbounded, we characterize conditions under which unboundedness will occur, and also compute the asymptotes that the unbounded solutions follow.  We demonstrate these results in trajectory coordination scenario modeled as a linear time varying quadratic game.
\end{abstract}

\section{Introduction}

Cooperation between humans and autonomous agents is of growing importance in a wide variety of dynamical systems, and is premised on the notion that humans place social expectations on all agents they interact with, including autonomous agents \cite{klein2024modeling,brown2024trash}. Consider, for example, the cooperation that is required in driving, such as in merging, changing lanes, coordination at stop signs, and other routine encounters.  Autonomous vehicles that cannot accommodate the need for cooperation with other vehicles can create confusion and even become safety hazards.  
However, within a cooperative framework, it may be important for agents to plan for contingencies.  Consider defensive driving, the mainstay of driver education, which recommends that drivers plan for worst-case scenarios to avoid unsafe interactions with other nearby drivers, while simultaneously abiding by rules and courtesies that represent an inherent notion of cooperation to navigate the roadways.  In essence, drivers plan for the worst case {\em within a cooperative framework}.  

We seek in this paper to formalize the notion of competition within a cooperative framework.  We are motivated by fundamental challenges in human-autonomy interaction that require sensitivity of the autonomous system to human expectations for cooperation, while simultaneously offering protection against the inherent unpredictability of human action and intent.  
We model cooperation through social value orientation 
\cite{liebrand1988ring, mcclintock1989social}, a framework from psychology that quantifies the degree to which an individual values another's benefit.  We use the premise of social value orientation to create coupled cost functions in a two-player game, and then evaluate the Nash equilibrium as a function of the degree of cooperation each player maintains towards the other player.  This results in a novel game theoretic framework, in which players seek a competitive equilibrium using costs that have an explicitly stated degree of cooperation.

Social value orientation was introduced to the robotics and controls community \cite{schwarting2019social}  as a means to predict human behavior in driving scenarios.  The authors inferred the degree of cooperation based on trajectories of human-driven vehicles in highway merging and in unprotected left turns, and then optimized the autonomous vehicle's trajectory via Nash and Stackelberg frameworks based on predictions of the trajectory of the human-driven vehicle.  It has also been used recently to design coordination policies between autonomouse vehicles (AVs). 
 Toghi, et al. \cite{toghi2021cooperative} develop a reinforcement learning framework for cooperative planning among AVs where agents receive rewards based on a linear combination of their own performance and their opponents performances.  This weighted combination that mimics social value orientation can be tuned to improve collective performance.

A variety of other techniques have been investigated in the literature to facilitate human-robot cooperation. Some work has focused on making human and robot actions predictable to each other.  Driggs-Campbell, et al. 
\cite{driggs2016communicating} identify nonverbal cooperation cues in human driving behavior and design autonomous control schemes that are more predictable to humans.  In 
\cite{driggs2017integrating},  Driggs-Campbell et al. 
empirical determine reachable sets of human drivers in order to accurately predict human driving behavior such as lane changing and use these reachable sets in an optimization based framework to generate trajectories for autonomous vehicles that mimic human behavior.

Some work has treated human intentions as a latent state that robots need to predict.  
Bajcsy, et al. 
\cite{bajcsy2017learning}
develop schemes for robots to learn human intentions in a physical optimal control setting by modeling the human state as hidden variables in a dynamical system that a robot can then estimate online and optimize accordingly. 
Tian et al. \cite{tian2022safety}
develop game theoretic models for humans and treat human rationality as an unobserved latent state that the robot has to infer. By inferring this state the robot can reduce conservatism while still maintaining safety in traffic scenarios.  
Sadigh et al.
\cite{sadigh2016planning,sadigh2018planning}
model the impact of an AV on a human driver as an underactuated dynamic system. The AV learns to estimate the human's objective via inverse reinforcement learning and then plans their optimal trajectory taking into account the human model (in a Stackelberg game framework) improving efficiency for the autonomous vehicle.  In practice the AV can end up learning to manipulate the human driver, cutting them off or blocking them in some cases or even leaving space to signal the human to go.

Other approaches have focused on designing automation actions that are legible to humans.  
Dragan et al. 
\cite{dragan2013legibility,dragan2014integrating,dragan2015effects} 
use optimal control techniques to design robot behaviors that specifically predictable and "legible" to a human in order to facilitate cooperation. 
Other techniques adjust centralized costs for autonomous agents to improve legibility by humans. \cite{kruse2012legible} 
and show that legibility of autonomous behavior can arise from rewarding efficiency in reinforcement learning in joint human robot tasks. 
\cite{busch2017learning} 
Some of these same approaches are useful not only for cooperation, but also for manipulation and deceipt Dragan \cite{dragan2015deceptive}, Sadigh
\cite{sadigh2016planning,sadigh2018planning}.






In contrast with these efforts, we seek to provide a formal analysis of how social value orientation shifts Nash-equilibria in a two-player games \cite{lambert2018quadratic}.  
Nash equilibium models 
\cite{nash1950non}\cite{bacsar1998dynamic} are ubiquitous in modeling interactions between agents in a variety of scenarios.  In optimal control contexts, Nash equilibria have been used to model competitive interactions between autonomouse agents such as autonomous racing 
\cite{liniger2019noncooperative}
\cite{zhu2024sequential} as well as collaborative coordination between 
humans and robots in manipulation and sensorimotor tasks \cite{chackochan2017modelling}. 
In particular, in competitive optimal control problems linear-quadratic games both in the open-loop and closed-loop settings have become a staple of modeling competitive interactions \cite{zou2020framework}. 
Many natural learning schemes also converge to Nash equilibria. \cite{frihauf2011nash} and 
Nash equilibrium models are often used in learning contexts for coordination between autonomous agents. \cite{li2016framework} 
in human robot interactions
\cite{music2020haptic}
\cite{an2023cooperative},
and even to understand human coordination behavior
\cite{lindig2021nash}.  

In our work, we show how the Nash equilibrium changes with different social value orientations of the agents. We show that unexpected behaviors can arise in even simple games, including unbounded equilibria.  This has significant implications for the design of cooperative autonomy.  The main contribution of this paper is 1) presentation of a theoretical framework for competition within an articulated level of cooperation for LQ games, 2) characterization of the equilibria of this game, and 3) analysis of these equilibria in relationship to equilibria from other related games.

The paper is structured as follows. In Section \ref{sec:svo}, we introduce the social value orientation in the context of two player games. In Section 
\ref{sec:nashbasics}, we review basic two player quadratic games and computation of Nash equilibria and other equilibrium points.  
We then introduce our main results in Section \ref{sec:svonash} in the form of several expansions that parametrize how the Nash equilibrium shifts under different social value orientations.  In Sections \ref{sec:contraction} and \ref{sec:blowups}, we then analyze the set of possible SVO-Nash equilibria in the cooperative case. 
 Specifically, we focus on two cases: 1) where the SVO-equilibria stay bounded and 2) where the SVO-equilibria can go unbounded.  In Section \ref{sec:contraction}, we give conditions for the first case in terms of ellipsoidal bounds on the equilibria.  We also detail some limited conditions on the costs where we can ensure that these bounded conditions are met (in Section \ref{sec:considerations}).  In Section \ref{sec:blowups}, we give conditions when the SVO-equilibria can go unbounded and give formulas for the asymptotic behavior.  In Section \ref{sec:openloopLQ}, we use our formulations to analyze the behavior of two-agents in a simple open-loop linear quadratic trajectory planning problem with collision avoidance. 
In Section \ref{sec:conclusion}, we give some further discussion, directions for future work, and concluding thoughts.  Section \ref{sec:appendices} contains Appendices with most of the proofs and extensions. 
Illustrative figures are included throughout with a focus on low-dimensional cases to build intuition.

\section{Social Value Orientation}

\label{sec:svo}

The social value orientation (SVO) of an agent in a two player game is a measure of how much an agents considers their opponents cost along with their own when they optimize.  It is usually captured as an angle $\theta_i \in [-\pi,\pi]$ that define a new cost for each player by taking a weighted combination of the player's own cost with their opponent's cost.  In a two player game with costs $J_i:\mathbb{R}^{d_i} \rightarrow \mathbb{R}$ for $i \in \{1,2\}$, if player $i$ has a social value orientation of $\theta_i$, then they optimize a new cost 
\begin{align}
\mathbf{J}_i(u;\theta_i) & = \cos(\theta_i) J_i(u) + 
\sin(\theta_i)
J_{-i}(u) 
\label{eq:mainSVOcosts}
\end{align}
We will use $\theta=(\theta_1,\theta_2)$ to denote the SVO values for the two players.  Note here that for player $i$, $\cos(\theta_i)$ gives the weight on their own cost and $\sin(\theta_i)$ gives the weight on the opponents cost.
The values of the angles $(\theta_1,\theta_2)$ determine different types of social behavior as illustrated in Fig. \ref{fig:svoregions}.   

In general, we can assume the SVO angle for each player is in the range $\theta_i \in [-\pi/2,\pi/2]$ since it would be abnormal for a player to want to increase their own cost (ie. $\theta_i \geq \pi/2$ or $\theta_i \leq -\pi/2$). In this paper, we will focus on the case where agents are trying to help each other to some extent, ie. $\theta_i \in [0,\pi/2]$, and we will look at how the Nash equilibrium shifts under different SVO values for these different levels of cooperation.

\begin{figure}
    \centering
\includegraphics[width=0.7\textwidth]{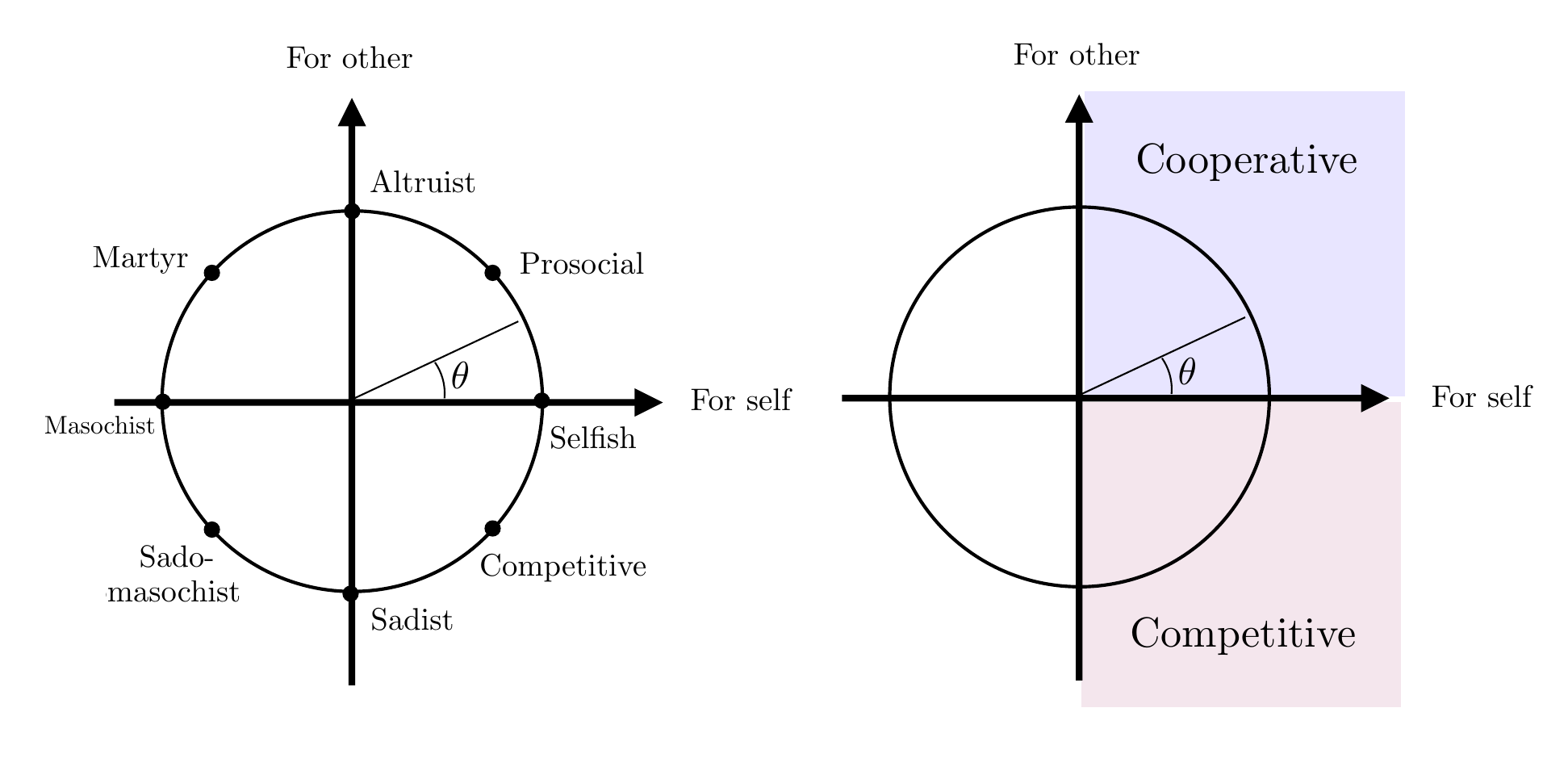}
    \caption{Social value orientation (SVO) region illustration with various types of personas labeled for context. SVO values in the upper-right orthant $\theta_i \in [0,\pi/2]$ define cooperative behaviors from selfish, $\theta_i = \pi/2$ to truly altruistic, $\theta_i = \pi/2$.  SVO values in the lower-right orthant $\theta_i \in [-\pi/2,0]$ define competitive behaviors from selfish, $\theta_i = 0$ to purely competitive, $\theta_i = -\pi/2$. }
    \label{fig:svoregions}
\end{figure}

\section{Nash Equilibria of Quadratic Games}

\label{sec:nashbasics}

Consider a two-player quadratic game where the action vectors for each player are $u_i \in \mathbb{R}^{d_i}$ for $i \in \{1,2\}$ with joint action vector denoted $u= (u_1,u_2) \in \mathbb{R}^d$ where $d=d_1+d_2$.  Each player's cost can be written as 
\begin{align}
J_i(u) = \tfrac{1}{2}u^\top M_i u + c_i^\top u 
\label{eq:origcosts}
\end{align}
with cost matrices $M_1,M_2 \in \mathbb{R}^{d \times d}$ 
\begin{align*}
M_1 & = \begin{bmatrix}
A_1 & B_1^\top \\
B_1 & D_1 
\end{bmatrix},
\quad 
M_2  = \begin{bmatrix}
D_2 & B_2 \\
B_2^\top & A_2
\end{bmatrix}, 
\end{align*}
and 
$c_1^\top = \big[ \ a_1^\top \ \ b_1^\top \big] \in \mathbb{R}^d $ and 
$c_2^\top = \big[ \ b_2^\top \ \ a_2^\top \big] \in \mathbb{R}^d $.  
The cost parameters are divided into the following subblocks. 
\begin{itemize}
    \item $A_i \in \mathbb{R}^{d_i \times d_i}$: Square matrices that give the quadratic dep. of each player's cost on their own actions
    \item $D_i \in \mathbb{R}^{d_{-i} \times d_{-i}}$: Square matrices that give the quadratic dep. of each player's cost on their opponent's action.
    \item $B_i \in \mathbb{R}^{d_{-i} \times d_i}$: 
    Possibly non-square matrices that give the bilinear dep. of each player's cost on both actions. 
    \item $a_i \in \mathbb{R}^{d_i}$: Linear dep. of the costs on ea. player's action
    \item $b_i \in \mathbb{R}^{d_{-i}}$:  Linear dependence on the opponent's action. 
\end{itemize}
Note the dimensions of each submatrix and the different organizations of $M_1$,$M_2$ and $c_1,c_2$.  This class of quadratic games includes open-loop linear-quadratic (LQ) games for linear time varying (LTV) dynamical systems.  We give further details on these in Section \ref{sec:openloopLQ}.  

We now list three possible sets of assumptions on $M_1$ and $M_2$ and their subblocks that impose different restrictions on the class of games.  Note that the game formulation is valid any of these cases.
\begin{assumption}
\label{assump:3cases}
Assumptions on $M_1,M_2$. (from least to most restrictive)
\begin{itemize}
    \item \textbf{Assump-1a:} $A_1,A_2 \succ 0$. This is the minimal assumption for the initial Nash equilibrium problem to be well-posed.  Many quadratic games (including the LQ games discussed later) will only satisfy this assumption and not the other two with $D_1,D_2$ indefinite. 
    \item \textbf{Assump-1b:} $A_1,A_2,D_1,D_2 \succ 0 $. This is the minimal assumption for the SVO-equilibrium problem (introduced in Sec. \ref{sec:svonash}) to be well-posed for all values of $(\theta_1,\theta_2) \in (0,\pi/2) \times (0,\pi/2)$.  
    \item \textbf{Assump-1c:} $M_1,M_2 \succ 0$.  This assumption is quite strong and requires the other two.
\end{itemize}
\end{assumption}
For the bulk of the paper, we will
assume the least restrictive assumption ($\textbf{Assump-1a}$) and will use the other two for reference and to make some specific arguments.

\subsection{Nash Equilibria: Basics}

The Nash equilibria of these quadratic games, $\mathbf{u}_{\text{N}} \in \mathbb{R}^d$, is computed by solving the first order optimality conditions for each player simultaneously, $\Big.\tfrac{\partial J_1}{\partial u_1} = 0$ and $\Big.\tfrac{\partial J_2}{\partial u_2} = 0$.
\begin{align*}
\begin{bmatrix}
u_1 \\ u_2
\end{bmatrix}^\top 
\begin{bmatrix}
A_1 \\ B_1
\end{bmatrix} + a_1^\top  = 0, \qquad 
\begin{bmatrix}
u_1 \\ u_2
\end{bmatrix}^\top 
\begin{bmatrix}
B_2 \\ A_2
\end{bmatrix} + a_2^\top  = 0
\end{align*}
Combined, this gives the system of equations 
\begin{align}
\begin{bmatrix}
\tfrac{\partial J_1}{\partial u_1} & \tfrac{\partial J_2}{\partial u_2} 
\end{bmatrix}
= 
u^\top \mathbf{M} + \mathbf{a}^\top = 0
\label{eq:equilibcond}
\end{align}
where 
\begin{align*}
\mathbf{M} = \begin{bmatrix}
A_1 & B_2 \\ B_1 & A_2 
\end{bmatrix}, \quad 
\mathbf{a}^\top = 
\begin{bmatrix}
a_1^\top & a_2^\top 
\end{bmatrix}
\end{align*}
with solution $\mathbf{u}_\text{N}^\top  = - \mathbf{a}^\top \mathbf{M}^{-1}$
under mild invertibility assumptions which will generally be satisfied since $A_1,A_2 \succ 0$ (except for particular $B_1,B_2$).  
Note that the partial derivatives are each with respect to only the player's own actions. Note also that several of the cost matrices/vector, specifically $D_1,D_2$ and $b_1,b_2$ do not show up in the calculation.  Along with the Nash equilibria, we can also compute several other important optimal points including the optimal for each player individually, $\mathbf{u}_1,\mathbf{u}_2 \in \mathbb{R}^d$,
\begin{align*}
\mathbf{u}_1^\top = - c_1^\top M_1^{-1}, \quad 
\mathbf{u}_2^\top = - c_2^\top M_2^{-1},
\end{align*}
the socially optimal or ``prosocial" solution, $\mathbf{u}_{\text{S}} \in \mathbb{R}^d$
\begin{align*}
\mathbf{u}_\text{S}^\top = - (c_1 + c_2)^\top \Big(M_1 + M_2\Big)^{-1}
=- (\mathbf{a} + \mathbf{b})^\top \Big(\mathbf{M} + \mathbf{N}\Big)^{-1}
\end{align*}
and a solution, we can call the \emph{altruistic Nash}, $\mathbf{u}_\text{A} \in \mathbb{R}^d$ where each player optimizes with respect to their opponents cost and we solve the system of equations.

\begin{align}
\begin{bmatrix}
\tfrac{\partial J_2}{\partial u_1} & \tfrac{\partial J_1}{\partial u_2} 
\end{bmatrix}
u^\top \mathbf{N} + \mathbf{b}^\top
= 0
\label{eq:equilibcond2}
\end{align}
where
\begin{align*}
\mathbf{N} = \begin{bmatrix}
D_2 & B_1^\top \\ B_2^\top & D_1
\end{bmatrix}, \quad 
\mathbf{b}^\top  = 
\begin{bmatrix}
b_1^\top & b_2^\top 
\end{bmatrix}
\end{align*}
with solution $\mathbf{u}_\text{N}^\top  = - \mathbf{b}^\top \mathbf{N}^{-1}$.  We will make the following assumption about the different matrices
\begin{assumption}
$M_1,M_2, \mathbf{M}, \mathbf{N}$ are invertible. 
\end{assumption}

We illustrate a basic quadratic game in 2D (where $d=2$ and $d_1=d_2=1$) in Fig. \ref{fig:nashillustration} by drawing the level sets of the costs and the various optimality curves as well as each of the optimal points listed above. 

\begin{remark}Note that each optimality curve gives the minimum of one of the costs when we restrict motion to either the $u_1$ or $u_2$ direction (and the other variable is fixed). Geometrically, the optimality curves pass through points where the tangents to the level sets point in either the $u_1$ (horizontal) or $u_2$ (vertical) directions.  Each equilibrium point, except the prosocial optimum, is at the intersection of two optimality curves. One can show that the prosocial optimum lies at the point where the level sets are parallel to each other. 
We will use this graphic for exposition through out the paper and to give a flavor for the results.  We note at the outset, however, that it is a particularly well-behaved example.  For example, the costs have ellipsoidal level sets. 
 This is the case only under $\textbf{Assump-1c}$, $M_1 \succ 0$ and $M_2 \succ 0$.   In general under \textbf{Assump-1a} or \textbf{Assump-2a}, the level sets may be saddle-shaped.
\end{remark}

\begin{figure}
    \centering    \includegraphics[width=0.6\textwidth]{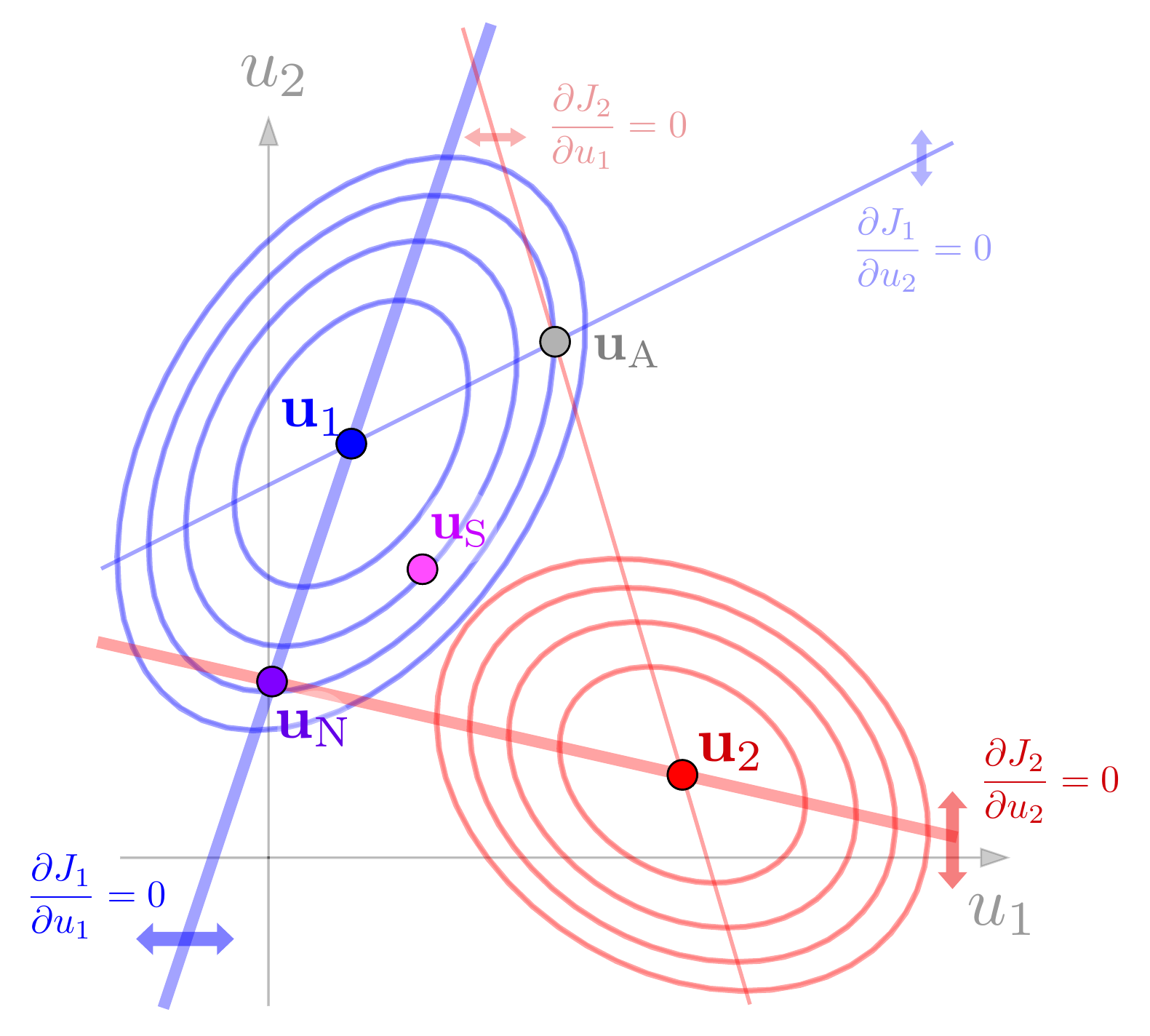}
    \caption{Illustration of equilibrium points for a 2-player scalar action quadratic game. Level sets and optimality curves for players 1 and 2 are shown in blue and red respectively.  Each optimality curve minimizes a cost function w.r.t motion in particular direction.  Note that the Nash eq. $\mathbf{u}_\text{N}$, the altruistic Nash $\mathbf{u}_\text{A}$, and the optima for players 1 $\mathbf{u}_1$ and 2 $\mathbf{u}_2$ are located at the intersections of various optimality curves. The socially optimal solution $\mathbf{u}_\text{S}$ is located at the point where the level sets of the costs are parallel. }
    \label{fig:nashillustration}
\end{figure}

\begin{figure}
    \centering    \includegraphics[width=0.49\textwidth]
    {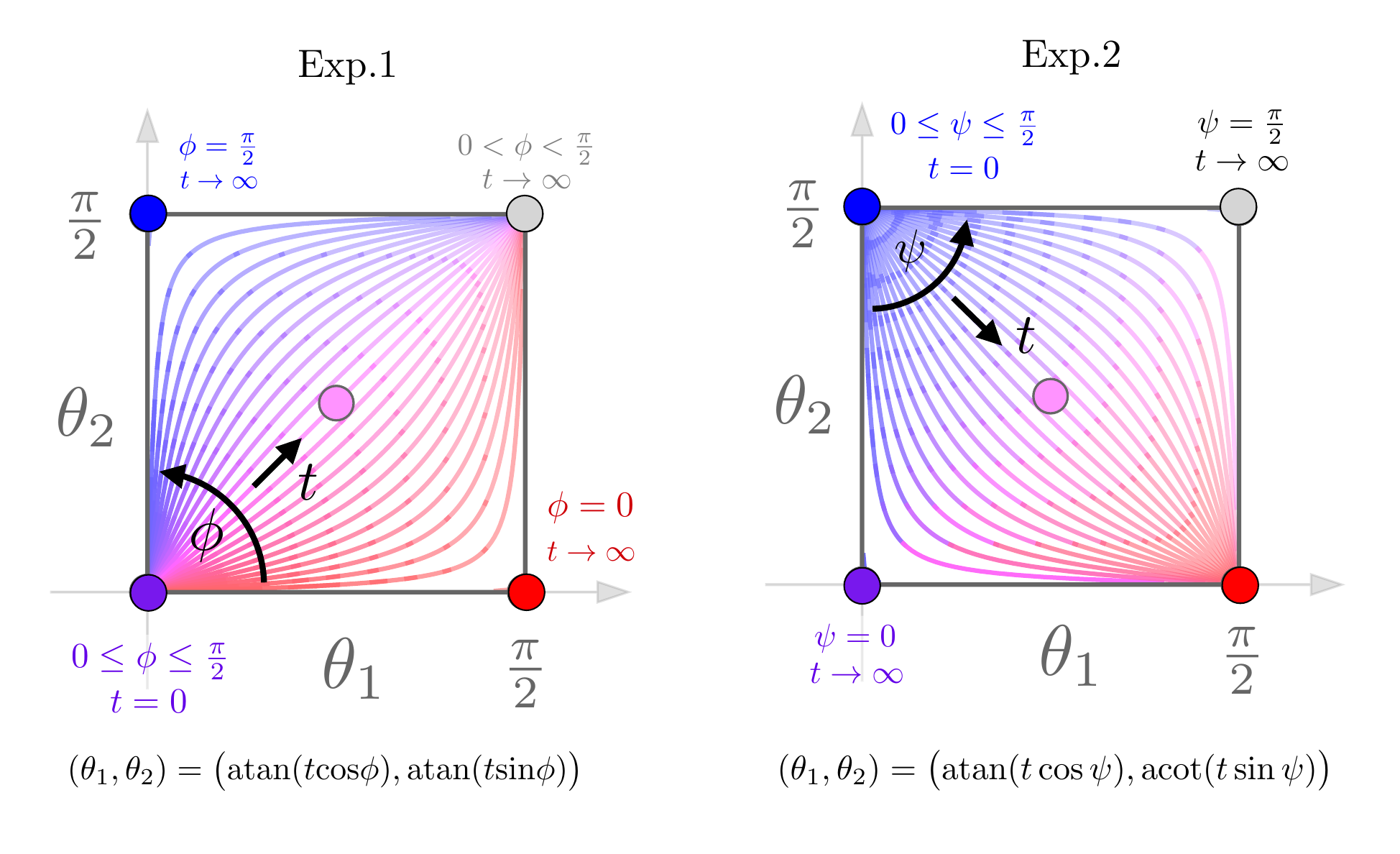}
    \includegraphics[width=0.49\textwidth]
    {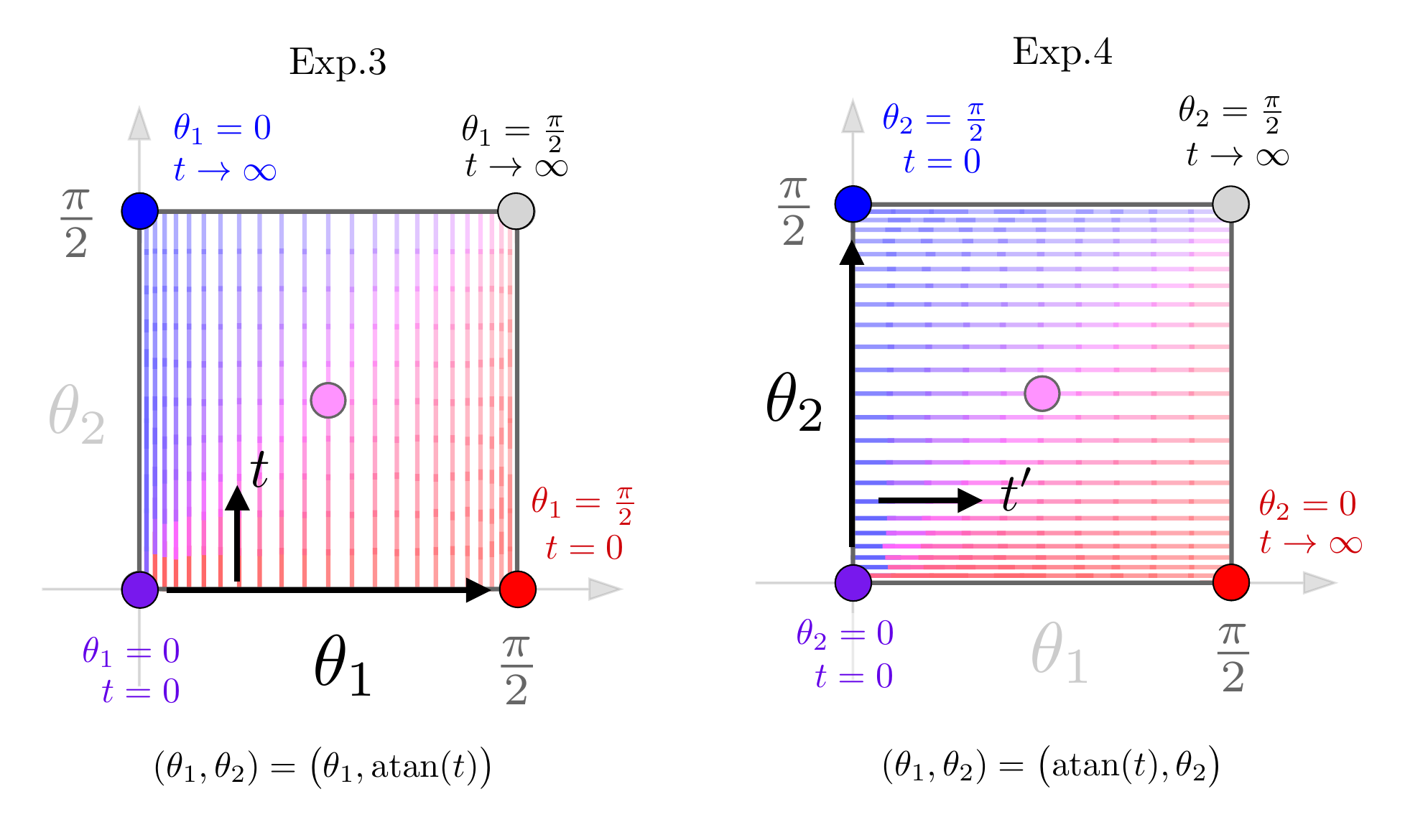}
    \caption{Alternative parametrizations of the $\theta \in [0,\pi/2] \times [0,\pi/2]$ space in terms of coordinates $(\phi,t)$ and $(\psi, t )$ for Exp.1 and Exp 2. respectively.  Note the limit points for different values of $\phi$ and $\psi$ as $t \to \infty$ and $ t  \to \infty$.}
    \label{fig:thetaparams}
\end{figure}





\section{SVO Nash Equilibria}

\label{sec:svonash}

For different values of the SVO, we can compute how the Nash equilibria will shift by resolving the optimality conditions with the new costs, that is by solving the system of equations
\begin{align*}
    \frac{\partial \mathbf{J}_1}{\partial u_1} = 0, \quad 
    \frac{\partial \mathbf{J}_2}{\partial u_2} = 0.
\end{align*}
These \emph{SVO-Nash equilibria} which we will denote $\mathbf{u}_\theta \in \mathbb{R}^d$ can be computed as 
\begin{align}
\mathbf{u}_{\theta}^\top 
& = 
-
\begin{bmatrix}
\text{c}\theta_1 a_1 + \text{s}\theta_1  b_2 \\ \text{c}\theta_2  a_2+ \text{s}\theta_2  b_1
\end{bmatrix}^\top  
\times
\begin{bmatrix}
\text{c}\theta_1 A_1 + 
\text{s}\theta_1  D_2 
& 
\text{c}\theta_2 B_2 + \text{s}\theta_2 B_1^\top  \\ \text{c}\theta_1 B_1+\text{s}\theta_1  B_2^\top & \text{c}\theta_2 A_2 + \text{s} \theta_2 D_1 
\end{bmatrix}^{-1} 
\label{eq:basicequilibinv}
\end{align}
where we have used $c\theta_i$ and $s \theta_i$ as shorthand for $\cos \theta_i$ and $\sin \theta_i$. We will refer to $\mathbf{u}_\theta$ as the \emph{SVO-Nash-equilibrium} or \emph{SVO-equilibrium} for a SVO-value or cooperation level of $\theta$.

\subsection*{Well-Posed SVO-Costs}

Note that for the game to be well-pose under the SVO transformation, the quadratic dependence of a player's cost on their own action must remain positive definite under different values of the SVO.  This amounts to the conditions 
\begin{align}
\cos\theta_i A_i + \sin \theta_i D_{-i} \succ 0, \quad \text{for} \quad i \in \{1,2\}
\label{eq:wellposed}
\end{align}
Note that $A_i \succ 0$ if the original Nash equilibrium is well-posed (\textbf{Assump-1a}).  If $D_{-i} \succ 0$ (\textbf{Assump-1b}), then the SVO-equilibrium problem is also well-posed for any $\theta_i \in (0,\pi/2)$.  When $D_{-i}$ is indefinite  there will be some cutoff $\bar{\theta}_i$ above which player $i$'s optimization is no longer well-posed, ie. when the left-hand side (LHS) of \eqref{eq:wellposed} becomes indefinite.  For the rest of the paper, we will assume $\theta_i \leq \bar{\theta}_i$.  These bounds on $\theta_i$ will have minimal effect on the proof techniques.  We will not explicitly state this assumption in the theorems though we will indicate this cutoff value when it is relevant in various figures.

To elucidate how the equilibrium shifts, we can expand the SVO-equilibrium computation in 
Eq. \ref{eq:basicequilibinv}
in several different ways.  In total, we will give four different expansions which we will denote the ``Nash-expansion" (\textbf{E1}), 
the ``Player Opt-expansion" (\textbf{E2}), the ``$\theta_1$-expansion" (\textbf{E3}), and the ``$\theta_2$-expansion (\textbf{E4}). 
The bulk of the paper will focus on the first two expansions ($\textbf{E1}$ and $\textbf{E2}$) since they prove more useful analytically.  Details of the remaining two expansions
($\textbf{E3}$ and $\textbf{E4}$) are given in the Appendix \ref{app:exp34} and \ref{app:proofs34}. 

\subsection{Coordinate Transformations}
In order to give these expansions, we introduce four coordinate transforms on the $(\theta_1,\theta_2) \in (0,\pi/2) \times (0,\pi/2)$ space in the following proposition.  
\begin{proposition}
\label{prop:coords}
The following coordinate transformations are bijections between $(\theta_1,\theta_2) \in (0,\pi/2) \times (0,\pi/2)$ and the domain given. 
\begin{align}
\textbf{E1:} \quad 
(\theta_1, \theta_2) & = \Theta_\phi(\phi,t)
= 
\Big(
\text{atan} \ t \cos \phi, 
\text{atan} \ t \sin \phi 
\Big) \label{eq:coords1}
\end{align}
for $(\phi, t ) \in 
(0,\pi/2) 
\times (0,\infty)$.
\begin{align}
\textbf{E2:} \quad 
(\theta_1, \theta_2) & = \Theta_\psi(\psi,t)
= 
\Big(
\text{atan} \ t \cos \psi, 
\text{acot} \ t \sin \psi
\Big) \label{eq:coords2}
\end{align}
for $(\psi, t ) \in 
(0,\pi/2) 
\times (0,\infty)$. 
\begin{align}
\textbf{E3:} \quad 
(\theta_1, \theta_2) & = \Theta_1(\theta_1,t)
= 
\big(\theta_1,\text{atan} \ t
\big) \label{eq:coords3}
\end{align}
for $(\theta_1, t ) \in 
(0,\pi/2) 
\times (0,\infty)$. 
\begin{align}
\textbf{E4:} \quad 
(\theta_1, \theta_2) & = \Theta_\psi(\psi,t)
= 
\big(\text{atan} \ t,\theta_2
\big) \label{eq:coords4}
\end{align}
for $(\theta_2,t) \in 
(0,\pi/2) 
\times (0,\infty)$. 
\end{proposition}
\begin{proof}
(See Appendix. \ref{app:coords12})
\end{proof}

These coordinate transformations are illustrated in Fig. \ref{fig:thetaparams}.
As illustrated in the figure, each coordinate transformation allows us to define a family of curves through the SVO value space where each curve is defined by $\phi$, $\psi$, $\theta_1$, and $\theta_2$ respectively and then the variable $t$ moves along the curve.  Each of the curves can be mapped into the action space $u \in \mathbb{R}^n$ by solving a specific eigenvalue problem.  
We detail these for expansions \textbf{E1} and \textbf{E2} in the next two sections.

\subsection{Nash-expansion (\textbf{E1})}



We now rewrite the SVO-Nash equilibria $\mathbf{u}_\theta$ for $\theta = (\theta_1,\theta_2)$ in terms of the new \textbf{E1} coordinates $(\phi,t) = \Theta_\phi(\theta)$.  Each $\phi$ defines a curve 
$
\Gamma_\phi(t): t \in \mathbb{R}_+ \mapsto u \in \mathbb{R}^d
$
and the SVO-equilibria $\mathbf{u}_\theta$ can be represented as a point along that curve. 


\begin{figure}
    \centering
\includegraphics[width=0.65\textwidth]{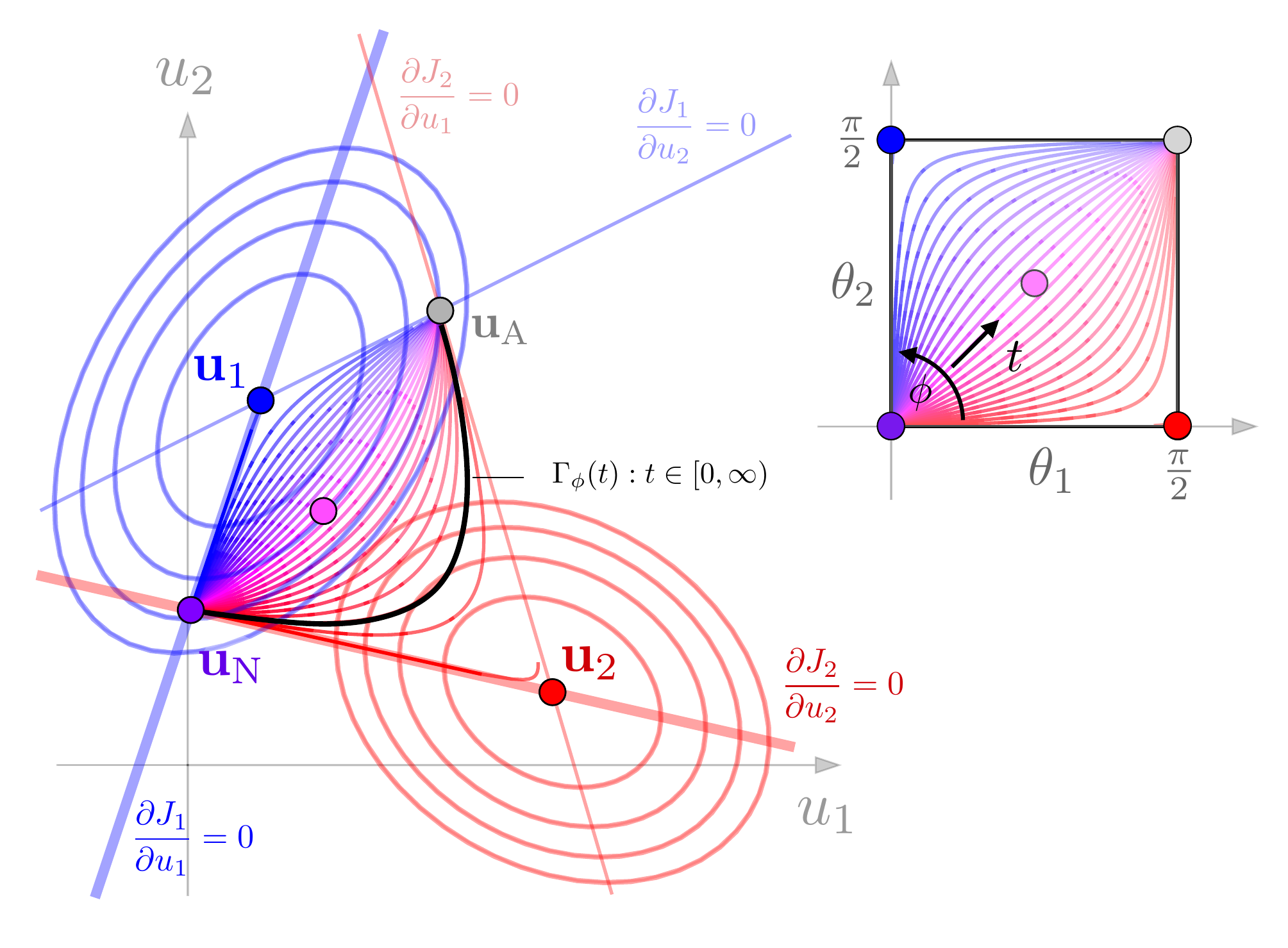}
    \caption{
    (\textbf{E1}) Nash-Expansion: The curves $\Gamma_\phi(t):t \in [0,\infty)$ for $\phi \in (0,\pi/2)$ are shown in color sweeping from $\mathbf{u}_\text{N}$ to $\mathbf{u}_\text{A}$. (Parameter values are given in App. \ref{app:values})}
    \label{fig:nashexp}
\end{figure}
\begin{proposition}[\textbf{E1:} Nash Expansion]
\label{prop:nashexp}
The SVO equilibria $\mathbf{u}_\theta$ 
for $(\theta_1,\theta_2) \in (0,\pi/2) \times (0,\pi/2)$ is given by 
\begin{align}
\mathbf{u}_\theta = 
\Gamma_\phi(t) = 
{\mathbf{u}_{\text{N}}}
+ 
G_\phi(t) 
\Big[{\mathbf{u}_{\text{A}}} -{\mathbf{u}_{\text{N}}}
\Big]
\label{eq:exp1main}
\end{align}
where
$(\phi,t) = \Theta_\phi^{-1}(\theta_1,\theta_2)$ and 
\begin{align}
G_\phi(t)
& = 
\Big[\tfrac{1}{t}\mathbf{M}
H_\phi^{\text{-}1}\mathbf{N}^{\text{-}1} + I \Big]^{-\top} 
=
\sum_i 
      W_{\phi i} 
\Big[
\tfrac{t}{\lambda_{\phi i}+t}
\Big]
V_{\phi i}^\top  
\label{eq:exp1sum} 
\end{align}
with 
$
H_\phi = \textbf{blkdg}
\big(\cos \phi I_{d_1 \times d_1}, \sin \phi I_{d_2 \times d_2}\big)$ where $\mathbf{M}H_\phi^{-1}\mathbf{N}^{-1}$ has eigenvalues $\Lambda_\phi = \{\lambda_{\phi i}\}_{i=1}^d$, right-eigenvectors 
$\{V_{\phi i} \}_{i=1}^d$ and left-eigenvectors 
$\{W_{\phi i}^\top \}_{i=1}^d$.
\end{proposition}
\begin{proof}
(See Appendix \ref{app:proofs})
\end{proof}
The curves $\Gamma_\phi(t)$ for the scalar action case are illustrated in Fig. \ref{fig:nashexp}.

In the above proposition, we have used the following characterization of the spectrum of $G_{\phi}(t)$ as 
\begin{align}
\text{spec}\Big(G_\phi(t)\Big)
= \Big\{ \Big(1+ \tfrac{\lambda_{\phi i}}{t} \Big)^{-1} \ \Big| \ \lambda_{\phi i} \in \text{spec}\big(\Lambda_{\phi}\big)\Big\}
\label{eq:exp1spec}
\end{align}

Intuitively, the Nash expansion \textbf{E1} writes an SVO Nash as the original Nash equilibrium $\mathbf{u}_\text{N}$ plus a correction term that is determined by transforming the difference between the altruistic Nash and the original Nash $\mathbf{u}_\text{A}-\mathbf{u}_\text{N}$ multiplied by the matrix $G_{\phi}(t)$ which varies with the social value orientation for each player.
\begin{figure}
    \centering
\includegraphics[width=0.65\textwidth]
{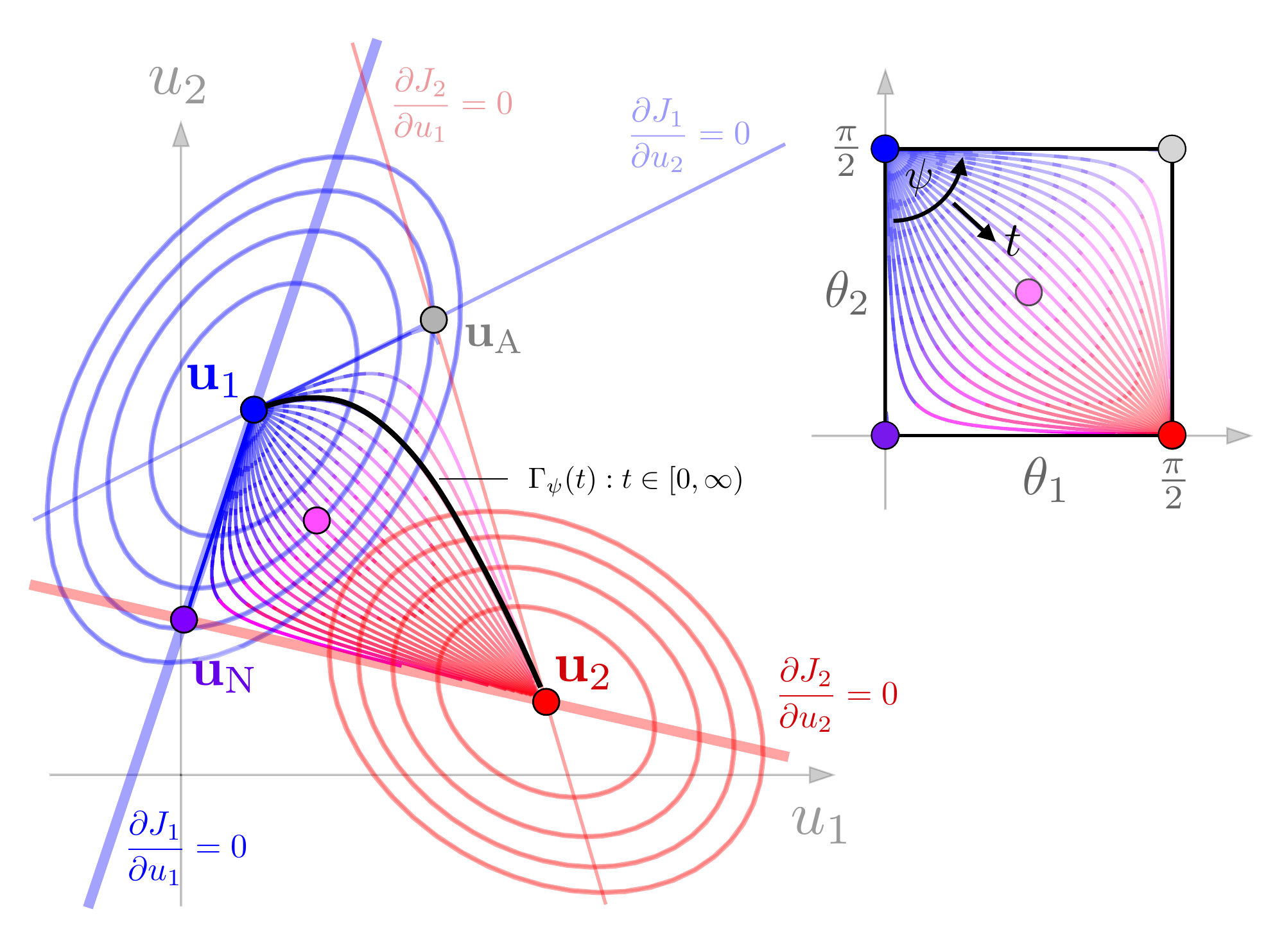}
    \caption{(\textbf{E2}) Player-Optimal-Expansion: The curves $\Gamma_\psi(t):t \in [0,\infty)$ for $\psi \in (0,\pi/2)$ are shown in color sweeping from $\mathbf{u}_1$ to $\mathbf{u}_2$.
    (Parameter values are given in App. \ref{app:values})
    }
    \label{fig:playerexp}
\end{figure}

\subsection{Player-opt expansion (\textbf{E2})}

We now state the Player-opt expansion as well.  In this case, we rewrite the SVO-Nash equilibria $\mathbf{u}_\theta$ for $\theta = (\theta_1,\theta_2)$ in terms of the new \textbf{E2} coordinates $(\psi,t) = \Theta_\psi(\theta)$.  Again, here $\psi$ defines a curve 
$
\Gamma_\psi(t): t \in \mathbb{R}_+ \mapsto u \in \mathbb{R}^d
$
and the SVO-equilibria $\mathbf{u}_\theta$ can be represented as a point along that curve. 

\begin{proposition}[\textbf{E2:} Player-Opt Expansion]
\label{prop:playerexp}
The SVO equilibria $\mathbf{u}_\theta$ 
for $(\theta_1,\theta_2) \in (0,\pi/2) \times (0,\pi/2)$ is given by 
\begin{align}
\mathbf{u}_\theta = 
\Gamma_\psi(t) = 
{\mathbf{u}_{1}}
+ 
G_\psi(t) 
\Big[{\mathbf{u}_{2}} -{\mathbf{u}_{1}}
\Big]
\label{eq:exp2main}
\end{align}
where $
(\psi,t) = \Theta_\psi^{-1}(\theta_1,\theta_2) 
$
\begin{align}
G_\psi(t)
& = 
\Big[\tfrac{1}{t}M_1
H_\psi^{\text{-}1}M_2^{\text{-}1} + I \Big]^{\text{-}\top} 
= 
\sum_i
W_{\psi i} 
\Big[
\tfrac{t}{\lambda_{\psi i}+t}\Big] 
V_{\psi i}^\top  
\label{eq:exp2sum} 
\end{align}
with 
$
H_\psi = \textbf{blkdg}
\big(\cos \psi I_{d_1 \times d_1}, \sin \psi I_{d_2 \times d_2}\big)$ where $M_1 H_\psi^{-1}M_2^{-1}$ has eigenvalues $\Lambda_\psi = \{\lambda_{\psi i}\}_{i=1}^d$, right-eigenvectors, 
$\{V_{\psi i} \}_{i=1}^d$ and left-eigenvectors 
$\{W_{\psi i}^\top \}_{i=1}^d$.
\end{proposition}


\begin{proof}
(See Appendix \ref{app:proofs})
\end{proof}
The curves $\Gamma_\psi(t)$ for the scalar action case are illustrated in Fig. \ref{fig:playerexp}.

In the above proposition, we have used the following characterization of the spectrum of $G_{\psi}(t)$ as 
\begin{align}
\text{spec}\Big(G_\psi(t)\Big)
= \Big\{ \Big(1+ \tfrac{\lambda_{\psi i}}{t} \Big)^{-1} \ \Big| \ \lambda_{\psi i} \in \text{spec}\big(\Lambda_{\psi}\big)\Big\}
\label{eq:exp2spec}
\end{align}

Intuitively, the Player-opt expansion \textbf{E2} writes an SVO Nash as the 
optimal for one player in this case player-1, $\mathbf{u}_\text{1}$, plus a correction term that is determined by transforming the fixed $\mathbf{u}_2-\mathbf{u}_1$ by the matrix $G_{\psi}( t )$.  Note that we could have easily swapped the roles of player 1 and player 2 (and the appropriate parameter matrices) in the above expansion. 

\begin{remark}
In the above two expansions, we have assumed that $\mathbf{M}H_\phi^{-1}\mathbf{N}^{-1}$ and $M_1 H_\psi^{-1} M_2^{-1}$ are diagonalizable. This is done for simplicity of presentation.  The results would be similar in the general Jordan case. 
\end{remark}

\begin{remark}
Before discussing the structure of the the set of SVO equilibria and the curves $\Gamma_\phi(t)$ and $\Gamma_{\psi}( t )$ further, we give a visualization in a 3D case where $d_1 = 2, d_2 = 1$ and $d=3$ shown in Fig. \ref{fig:3dexample}.  This example provides several pieces of intuition immediately 
First, while still a two dimensional surface in the higher dimensional space, the space of SVO equilibria is only flat in 2D; in general higher dimensional spaces it can be curved (and even disconnected as we will see in later examples).   
 Second, the set of SVO equilibria does not necessarily lie in the convex hull of ${\mathbf{u}_\text{N}},{\mathbf{u}_\text{A}},{\mathbf{u}_1},{\mathbf{u}_2}$ in higher dimensions as it appears to in the 2D examples in Figs. \ref{fig:nashexp} and \ref{fig:playerexp}.  The full picture of the geometry of this set can be much more surprising, a fact that we explore in the next section. 
 \end{remark}

\begin{figure}
    \centering
\includegraphics[width=0.65\textwidth]
 {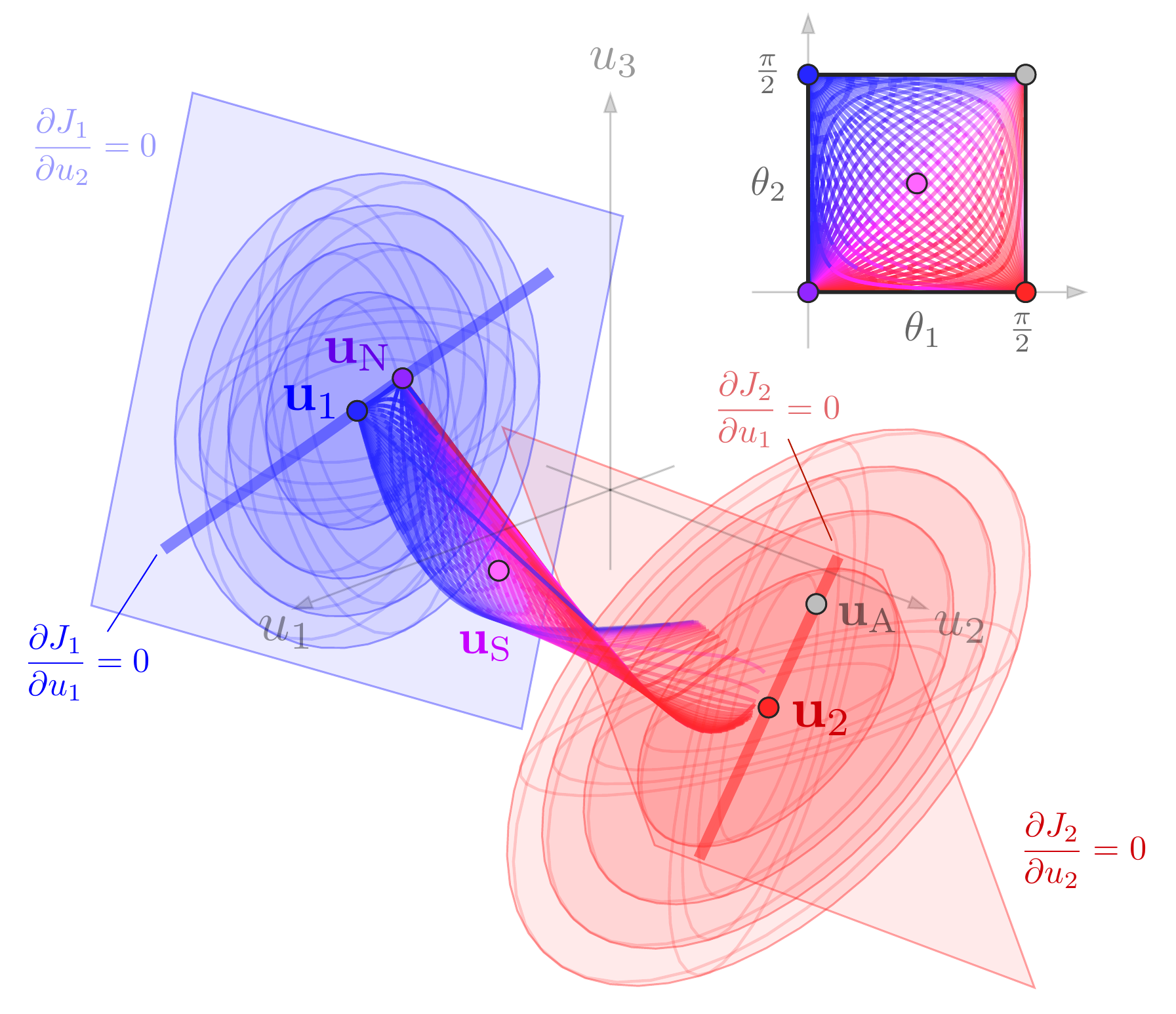}
    \caption{Illustration of the set of SVO-equilibria for 3D example, with $u_1 \in \mathbb{R}^2$ and $u_2 \in \mathbb{R}$.  Note that from this example we can see clearly that the space of SVO equilibria is not flat and also not in the convex hull of the equilibrium points.
    (Parameter values are given in App. \ref{app:values}.)
    }\label{fig:3dexample}
\end{figure}

\section{Set of SVO Equilibria}

In this section, we make comments about the geometry of the set of SVO equilibria. The geometry of each of these sets can be divided into two cases depending on the spectra of $G_\phi(t)$ (and $G_\psi(t)$).
\begin{itemize}
    \item Case 1: $G_{\phi}(t),G_{\psi}(t)$
    are contractions
    \item Case 2: 
$G_{\phi}(t),G_{\psi}(t)$
blow-up for at finite points $t \in [0,\infty)$.
\end{itemize}

\subsection{
Case 1: Bounded SVO Equilibria
}

\label{sec:contraction}

We start with the bounded case.  
For the following arguments define matrices of eigenvectors 
\begin{align*}
V_\phi = \Big[ V_{\phi 1} \ \cdots \ V_{\phi d}  \Big], \quad 
V_\psi = \Big[ V_{\psi 1} \ \cdots \ V_{\psi d}  \Big]
\end{align*}
When $\text{spec}(\Lambda_\phi) > 0$, we can show that the SVO equilibria along the curve $\Gamma_\phi(t)$ or $\Gamma_\psi(t)$ lie within a union of two ellipsoids.  
To make this argument, we will define positive-definite norms from the eigenvectors $V_\phi$ and $V_\psi$ and show that the maps $G_\phi(t)$ and $G_\psi(t)$ are contractions with respect to either of these norms for all $t \in [0,\infty)$.  
The constructions are parallel so we will only show the initial construction here for $V_\phi$ and $G_\phi(t)$.  

\subsubsection{$P_\phi$-Norm Contractions}
 Define the positive definite matrix $P_\phi = V_\phi V_\phi^\top$ 
 %
 Consider the following inner product and induced ``$P_\phi$-norm"
\begin{align}
\langle v,u \rangle_{P_\phi} :=
v^\top P_\phi u, \qquad 
\vert\vert u \vert\vert_{P_\phi} := 
\big(u^\top P_\phi u\big)^{\tfrac{1}{2}}
\end{align}
Define the open $P_\phi$-norm ball, $\mathcal{B}_\phi(c,r)$around center $c$ with radius $r>0$ 
\begin{align*}
\mathcal{B}_\phi(c,r) = \Big\{
u \in \mathbb{R}^{d} \ \big| \ 
\vert\vert u -c\vert\vert_{P_\phi} < r
\Big\} 
\end{align*}
and let $\mathcal{\bar B}_\phi(c,r) = \Big\{
u \in \mathbb{R}^{d} \ \big| \ 
\vert\vert u -c\vert\vert_{P_\phi} = r
\Big\} $ denote the boundary.  
Note that the boundaries of $\mathcal{B}_\phi(c,\cdot)$ are ellipsoidal level sets of the quadratic form $$(u-c)^\top P_\phi (u-c) = (u-c)^\top V_\phi V_\phi^\top (u-c)$$. We can now give the following definition of a contraction with respect to the $P_\phi$ norm. 
\begin{definition}
\label{def:contract}
A matrix $G$ is a contraction with respect to the $P_\phi$-norm if for any vector $u$ with $\vert\vert u \vert\vert_{P_\phi} = r$, then $\vert\vert Gu \vert\vert_{P_\phi} < r$.  Equivalently, if $u \in \mathcal{\bar B}_\phi(c,r)$, then $Gu \in \mathcal{B}_\phi(c,r)$. 
\end{definition}
We can now state the following lemma. \begin{lemma}
\label{lem:contract}
$G_\phi(t)$ is a contraction w.r.t. the $P_\phi$-norm 
for all $t \in [0,\infty)$ if and only if 
$\text{spec}(\Lambda_\phi) > 0$.
\end{lemma}
\begin{proof}
Based on the characterizations of the spectra of $G_\phi(t)$ in Eqns.~\eqref{eq:exp1spec}
for $t \in [0,\infty)$ we have that 
\begin{align*}
\big|\text{spec}\big(G_\phi(t)\big)\big| < 1 \ \ \iff \ \ \text{spec}\big(\Lambda_\phi \big) > 0
\end{align*}
Note that $G_\phi(t) = V_\phi^{-\top} D_{\phi,t} V_\phi^\top$ where $D_{\phi,t}$ is a diagonal matrix with elements defined by the eigenvalues given in Eq.~\eqref{eq:exp1spec}.  Consider $u$ with $\vert\vert u\vert\vert_{P_\phi} = r$ and note that
\begin{align*}
\vert\vert 
G u\vert\vert_{P_\phi}^2
& = 
u^\top V_\phi D_{\phi,t}^2 V_\phi^\top u <
u^\top V_\phi V_\phi^\top u = r^2
\end{align*}
\end{proof}
From this contraction property, we can get a course but somewhat surprising bound on each curve defined by $G_{\phi}(t)$ and $G_\psi(t)$
\begin{proposition}
\label{prop:intersectellipses}
Given $\text{spec}(\Lambda_\phi) >0$, 
\begin{subequations}
\begin{align}
\Gamma_{\phi}(t) & \subset \mathcal{B}_\phi \big({\mathbf{u}_\text{N}},r\big) \cap
\mathcal{B}_\phi \big({\mathbf{u}_\text{A}},r\big) 
\label{eq:balls1}
\end{align}
for all $t \in [0,\infty)$
with $r =
\vert\vert 
{\mathbf{u}_\text{A}} - {\mathbf{u}_\text{N}} 
\vert\vert_{P_\phi}$.  Similarly, if $\text{spec}(\Lambda_\psi) > 0$, then 
\begin{align}
\Gamma_{\psi}( t ) 
 & \subset \mathcal{B}_{\psi}\big({\mathbf{u}_1},r'\big) \cap
\mathcal{B}_{\psi}\big({\mathbf{u}_2},r'\big) 
\label{eq:balls2}
\end{align}
for all $ t  \in [0,\infty)$ with $r' =
\vert\vert 
{\mathbf{u}_1} - {\mathbf{u}_2}
\vert\vert_{P_\phi}$
\end{subequations}
\end{proposition}
\begin{proof}
We only discuss \eqref{eq:balls1} since the proof of \eqref{eq:balls2} is parallel. The inclusion of $\Gamma_\phi(t)$ in $\mathcal{B}_\phi(\mathbf{u}_\text{N},r)$ is immediate from noting the form of Eqn. \eqref{eq:exp1main} and applying Lemma \ref{lem:contract}.  
Inclusion in $\mathcal{B}_\phi(\mathbf{u}_\text{A},r)$ comes from noting that 
reversing the roles of $\mathbf{u}_\text{N}$ and  $\mathbf{u}_\text{A}$ and also $\mathbf{M}$ and $\mathbf{N}$ and 
following the same argument as the Nash expansion, we can write 
\begin{align}
\mathbf{u}_\theta 
& = 
{\mathbf{u}_{\text{A}}}
+ 
\underbrace{
\Big(t\mathbf{N}
H_\phi \mathbf{M}^{-1} + I \Big)^{-\top}}_{G'_\phi(t)}  
\Big[{\mathbf{u}_{\text{N}}} -{\mathbf{u}_{\text{A}}}
\Big] \label{eq:exp1main_alt} 
\end{align}
We note that $t\mathbf{N}H_\phi \mathbf{M}^{-1} = \Big(\tfrac{1}{t}\mathbf{M}H_\phi^{-1}
\mathbf{N}^{-1}\Big)^{-1}$ and draw several implications.  First the spectrum of 
$t\mathbf{N}H_\phi \mathbf{M}^{-1}$
is given by $\text{spec}(\Lambda_\phi)^{-1}$ and $\spec (\Lambda_\phi) > 0$ implies that $\text{spec}(\Lambda_\phi)^{-1} > 0$.  Second $G_\phi(t)$ and $G'_\phi(t)$ have the same eigenvectors.  From Lemma \ref{lem:contract} then, we have that $G'_\phi(t)$ is a $P_\phi$-norm contraction and that $\Gamma_\phi(t)\subset \mathcal{B}_\phi(\mathbf{u}_\text{A},r)$ which completes the argument.

\end{proof}

\begin{figure}
    \centering
    \includegraphics[width=0.8\textwidth]{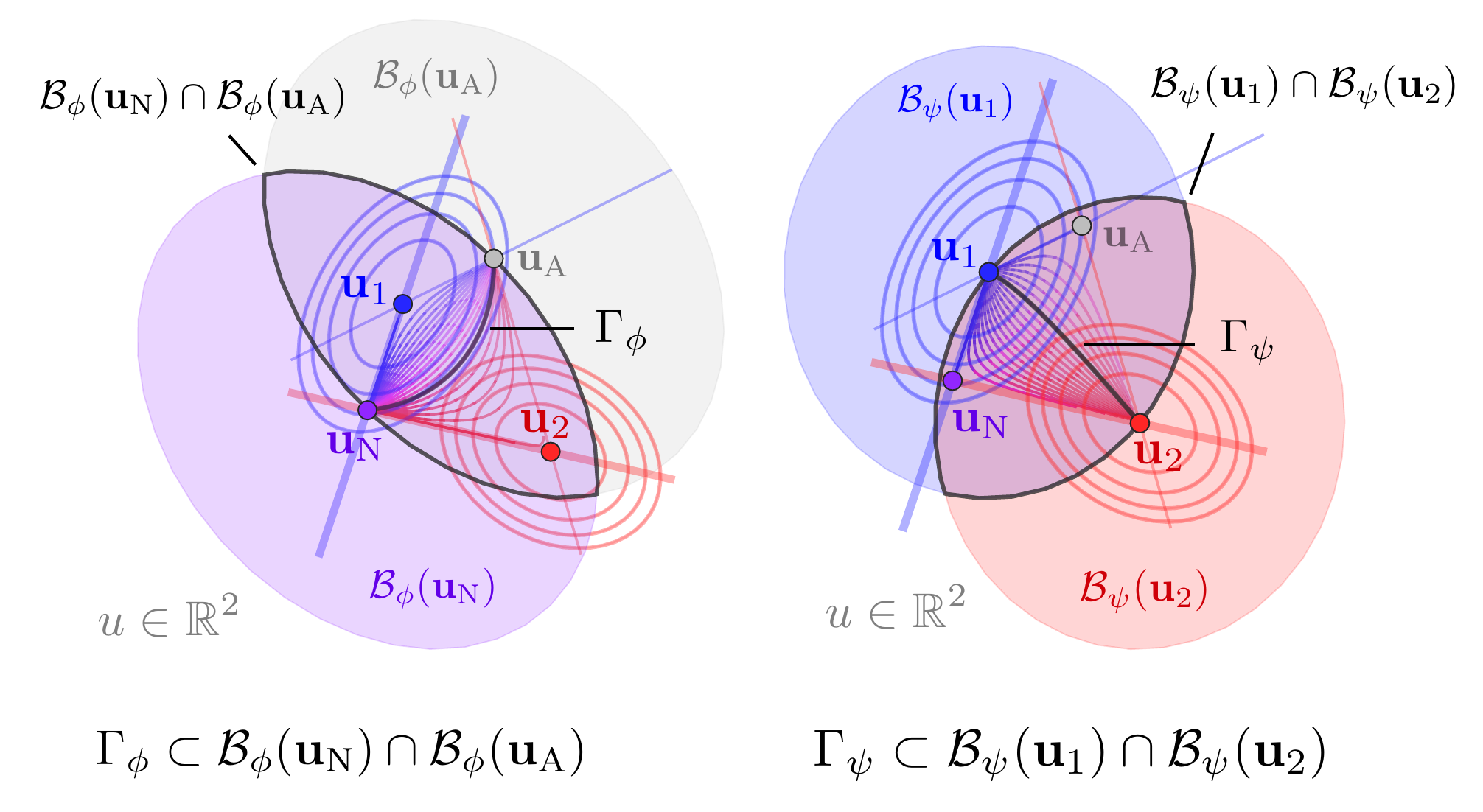}
    \caption{
    (Left) Ellipsoidal bounds on $\Gamma_\phi(t)$ for $t \in [0,\infty)$ when $\text{spec}(\Lambda_\phi) > 0$ from Eq. \eqref{eq:balls1}.  
    (Right) Ellipsoidal bounds on $\Gamma_\psi(t)$ for $t \in [0,\infty)$ when $\text{spec}(\Lambda_\psi) > 0$.
    Note that the dependence on $r,r'$ and $t$ are suppressed for clarity.
    Parameter values are given in App. \ref{app:values}.
    }
    \label{fig:balls}
\end{figure}
This proposition is illustrated for a simple $2\times 2$ example in Fig. \ref{fig:balls}. If we want to go further and provide bounds on a specific SVO equilibrium $\mathbf{u}_\theta$, we can use both bounds provided in 
Prop. \ref{prop:intersectellipses}.  We state this here as a proposition without proof since it is immediate from Prop.  \ref{prop:intersectellipses} given that $\Theta_\phi$ and $\Theta_\psi$ are bijections. 
\begin{proposition}
\label{prop:intersectellipses4}
Given $\text{spec}(\Lambda_\phi) >0$ and $\text{spec}(\Lambda_\psi) > 0$
\begin{align}
    \mathbf{u}_\theta \in \mathcal{B}_\phi \big({\mathbf{u}_\text{N}},r\big) \cap
\mathcal{B}_\phi \big({\mathbf{u}_\text{A}},r\big) 
\cap
\mathcal{B}_{\psi} \big({\mathbf{u}_1},r'\big) \cap
\mathcal{B}_{\psi}
\big({\mathbf{u}_2},r'\big) 
\label{eq:intersect4}
\end{align}
with $r =
\vert\vert 
{\mathbf{u}_\text{A}} - {\mathbf{u}_\text{N}} 
\vert\vert_{P_\phi}$.
and $r' =
\vert\vert 
{\mathbf{u}_1} - {\mathbf{u}_2}
\vert\vert_{P_\psi}$
where $(\phi,\cdot) = \Theta_\phi^{-1}(\theta)$ and $(\psi,\cdot) = {\Theta_{\psi}}^{-1}(\theta)$. 
\end{proposition}
Prop. \ref{prop:intersectellipses4} is illustrated in Fig. \ref{fig:ellipsesNA12}.  
\begin{figure}
    \centering
    \includegraphics[width=0.45\textwidth]
    {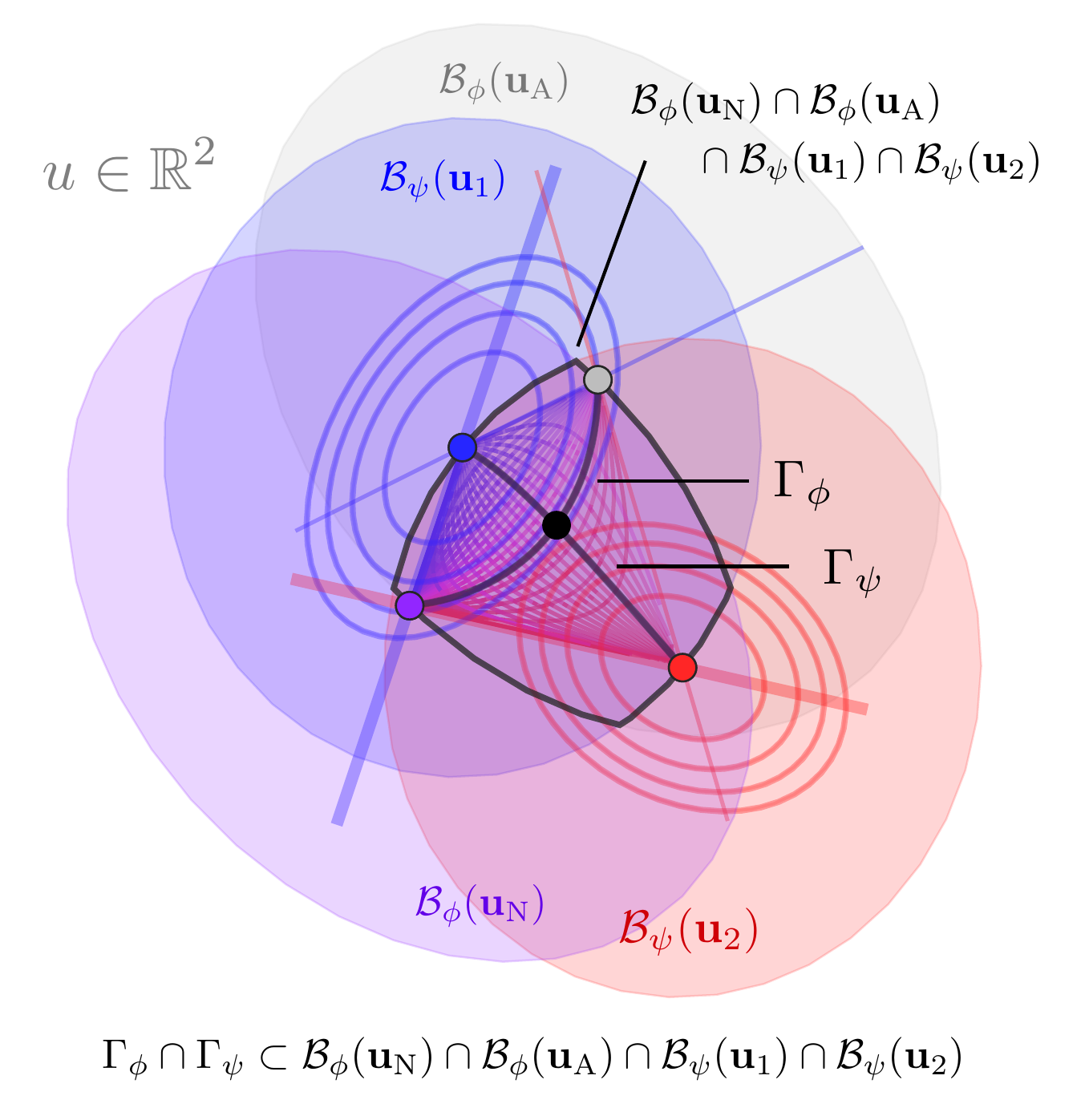}
    \caption{Joint ellipsoidal bounds on $\mathbf{u}_\theta = \Gamma_\phi(t) \cap \Gamma_\psi(t)$ for $(\phi,t) = \Theta_\phi^{-1}(\theta)$ and 
    $(\psi,t) = \Theta_\psi^{-1}(\theta)$ when $\text{spec}(\Lambda_\phi) > 0$ and $\text{spec}(\Lambda_\psi) > 0 $
    from Eq. \eqref{eq:intersect4}. 
    Note that the dependence on $r,r'$ and $t$ are suppressed for clarity.
    Parameter values are given in App. \ref{app:values}.
    }
    \label{fig:ellipsesNA12}    
\end{figure}
Note that for different values of $(\theta_1,\theta_2)$  and thus $(\phi,t)$ and $(\psi, t )$ the shape of the $P_\phi-\text{norm}$ and $P_{\psi}-\text{norm}$ balls (and their intersections change).  We illustrate for several sample SVO values in Fig. \ref{fig:ellipsessamples}.

\begin{figure}
    \centering
    \includegraphics[width=0.6\textwidth]
    {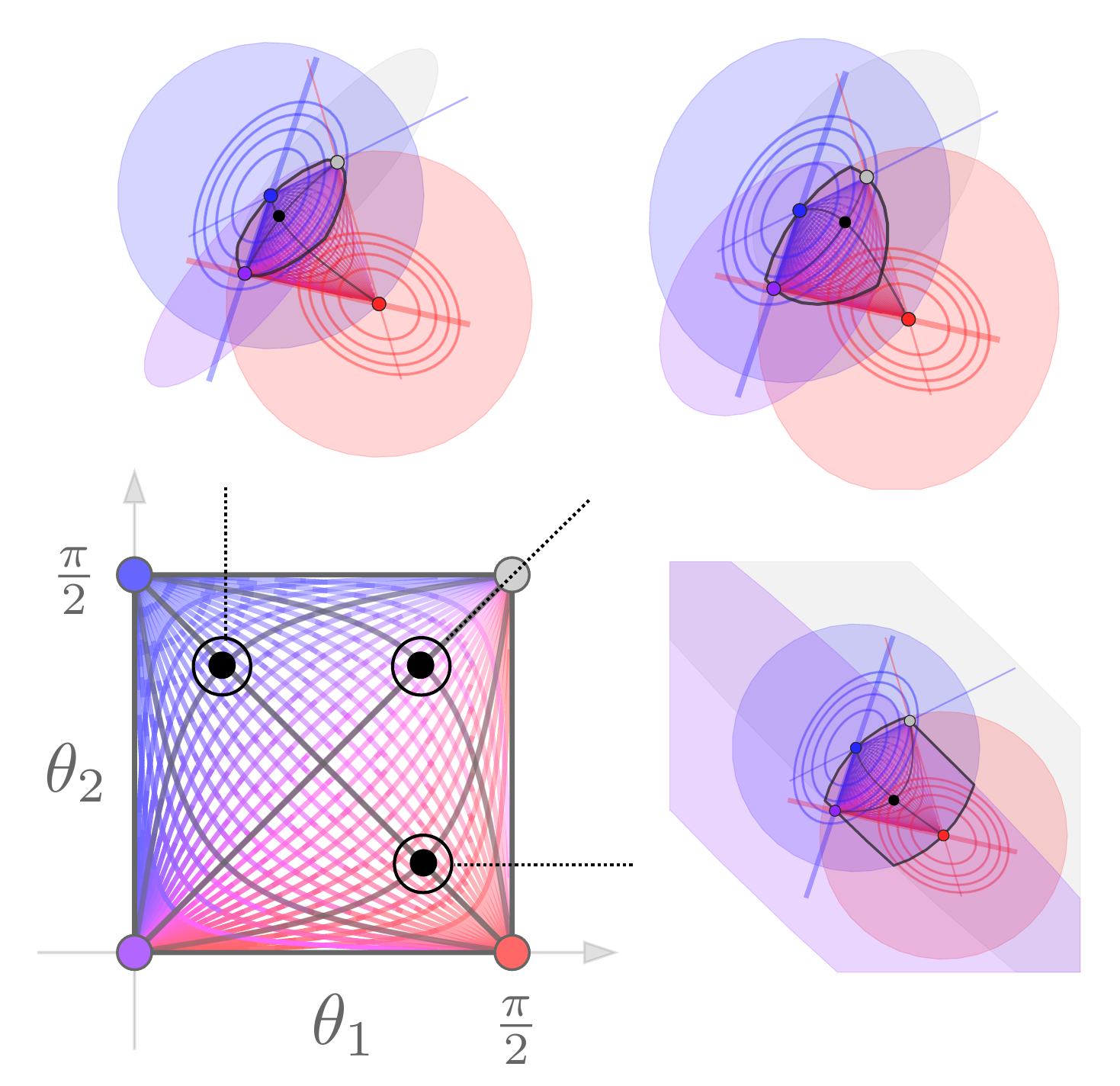}
    \caption{
    For different values of $\theta$, the corresponding values of $\phi$ and $\psi$ vary and thus the shapes of the $P_\phi$ and $P_\psi$-balls vary as well.  The resulting shapes of the ellipsiodal bounds and their intersections is illustrated for several different values of $\theta$. Clockwise from the top-left, $\theta = (\pi/8,3\pi/8)$, 
    $\theta = (3\pi/8,3\pi/8)$, and 
    $\theta = (3\pi/8,3\pi/8)$.  
    (Parameter values are given in App. \ref{app:values}.)}\label{fig:ellipsessamples}    
\end{figure}

\subsubsection{Eigenvalue Considerations}

\label{sec:considerations}

We now give a a few arguments about how the eigenvalues will behave in a limited number of scenarios.  In particular, we focus on Expansion 2 and the eigenstructure of the matrix $M_1 H_\psi^{-1} M_2^{-1}$.  In this expansion the matrices $M_1,M_2$ are the quadratic portion of the costs for each agent and are thus perhaps most easy to relate to the structure of the game.  
Here, we will provide a bounding result for \textbf{Assump-1c} and then comment why it may be difficult to obtain bounding results given only \textbf{Assump-1a} \& \textbf{Assump-1b}. 
Given that costs $M_1,M_2$ satisfy \textbf{Assump-1c}, we can make the following limited statement about the eigenvalues in Expansion 2.  
\begin{proposition}
\label{prop:M1M2posdef}
If $M_1,M_2 \succ 0$ then for $\psi = \pi/4$, 
\begin{align*}
\text{spec}\Big(\tfrac{1}{\sqrt{2}}M_1H_\psi^{-1}M_2^{-1}\Big) 
=
\text{spec}\Big(M_1M_2^{-1}\Big) 
> 0
\end{align*}
\end{proposition}

\begin{proof}
Since $\psi = \pi/4$, we have that $\tfrac{1}{\sqrt{2}}M_1 H_\psi^{-1} M_2^{-1} = M_1 M_2^{-1}$. Since $M_1,M_2 \succ 0$, we can take $M_2^{-1/2}$ to be the unique, positive definite symmetric square root of $M_2^{-1}$. We then have the following similarity relation (denoted $\sim$)
\begin{align*}
M_1M_2^{-1} \sim M_2^{-1/2}\big(M_1M_2^{-1}\big)M_2^{1/2}
=
M_2^{-1/2}M_1M_2^{-1/2}
\end{align*}
The RHS is congruent to $M_1$ and thus 
\begin{align*}
\text{spec}\Big(M_2^{-1/2}M_1M_2^{-1/2}\Big) > 0 \iff 
\text{spec}\Big(M_1 \Big) > 0 
\end{align*}
Thus since $M_1 \succ 0$, we have that $\text{spec}(M_1M_2^{-1}) >0$.  
\end{proof}
This proposition implies the following corollary.
\begin{corollary}
If $M_1,M_2 \succ 0$, the prosocial equilibrium $(\theta_1,\theta_2) = (\pi/4,\pi/4)$ is inside the intersection bound given in Eq. 
\ref{eq:balls2} with 
$\psi = \pi/4$. 
\end{corollary}

\begin{remark}
Prop. 
\ref{prop:M1M2posdef} is quite restrictive in its assumptions and thus limited in its application; however, more general characterization of $\text{spec}(\Lambda_\psi)$ based on $M_1,M_2$ may be difficult.  Even if we just limit ourselves to the curve containing the prosocial solution ($\psi=\pi/4$, $H_\psi$ proportional to the identity), we can find symmetric $M_1,M_2$ such that $M_1M_2^{-1}$ is \emph{any} desired matrix \cite{bosch1986factorization}.  Neither \textbf{Assump-1a} nor \textbf{Assump-1b} guarantee anything about the spectra of $M_1$ or $M_2$ and thus it is not clear that we can state anything further about the spectrum of $M_1M_2^{-1}$.  Moving away from the case where $H_\psi$ is proportional to the identity complicates the matter even further. Any further guarantees on the eigenvalues that could be found could be quite useful given the pathological nature of the unbounded solutions caused by negative eigenvalues as detailed in the next section. 
There is much room for the industrious reader to go beyond these arguments to make more strict guarantees on the eigenvalues depending on the structure of the matrices $M_1,M_2,\mathbf{M}$ and $\mathbf{N}$ in their individual problems. \end{remark}





\subsection{Case 2: Unbounded SVO Equilibria}

\label{sec:blowups}

\subsubsection{Finite Blow-up Points}
In the case where $\text{spec}(\Lambda_\phi)\ngtr 0$ and $\text{spec}(\Lambda_\psi)\ngtr 0$, the SVO-equilibria curves $\Gamma_\phi(t)$ and $\Gamma_\psi(t)$ can blow up for a discrete number of finite values $t \in [0,\infty)$. 
 To quantify where and how these blow-ups occur, we state the following proposition that focuses on asymptotic behaviors of the Nash-expansion (\textbf{E1}), Eq. \eqref{eq:exp1sum}.  
Parallel characterizations can be written for the asymptotic behavior of Expansion \textbf{E2} in Prop. \ref{prop:playerexp} and Expansions \textbf{E3} and \textbf{E4} in Props. \ref{prop:theta1exp} and \ref{prop:theta2exp} in Appendix \ref{app:proofs34}.

 We use the notation $\text{spec}^-(\Lambda_\phi)$ to denote the negative eigenvalues of $\Lambda_\phi$. 
 \begin{proposition}
 \label{prop:asymp}
 Suppose $\text{spec}^-(\Lambda_\phi)$ is non-empty. 
 Let $\lambda_{\phi j} \in \text{spec}^-(\Lambda_\phi)$ have multiplicity denoted by the index set $\mathcal{J}$. The asymptotic behavior of $\Gamma_\phi(t)$ as $t\to |\lambda_{\phi j}|$ can be written as 
 \begin{subequations}
 \label{eq:asymp}
\begin{align}
\lim_{t\to |\lambda_{\phi j}|}
\Gamma_\phi(t) & = 
{{\mathbf{u}_\text{N}}} + \mathbf{u}_{\phi j}^{\text{fin}} + 
\lim_{t\to |\lambda_{\phi j}|}
\left(
1+\tfrac{\lambda_{\phi j}}{|\lambda_{\phi j}|} \right)^{-1}
\mathbf{u}_{\phi j}^{\text{inf}} \label{eq:exp1asymp} \\
\text{with} \ \ 
\mathbf{u}_{\phi j}^{\text{fin}} & = \sum_{i\notin \mathcal{J} } 
       W_{\phi i} 
\Big(1+
\tfrac{\lambda_{\phi i}}{|\lambda_{\phi j}|}\Big)^{-1}
V_{\phi i}^\top 
\Big[{{\mathbf{u}_\text{A}}} -{{\mathbf{u}_\text{N}}}
\Big]
\notag \\
\mathbf{u}_{\phi j}^{\text{inf}} & = 
\sum_{j \in \mathcal{J} } 
    W_{\phi j} 
V_{\phi j}^\top
\Big[{{\mathbf{u}_\text{A}}} -{{\mathbf{u}_\text{N}}}
\Big]
\notag
\end{align}
\end{subequations}
\end{proposition}
The limits in Prop. \ref{prop:asymp} give an explicit characterization of how the SVO equilibria blows up as $t \to |\lambda_{\phi j}|$ for $\lambda_{\phi j}<0$ and thus $(\lambda_{\phi j}/|\lambda_{\phi j}| + 1)\to 0$. Note that since $t \in [0,\infty)$ these limit points are only a problem for any eigenvalues 
$\lambda_{\phi j}$ that are real and negative. 
 Since there are a finite number of eigenvalues, the curve $\Gamma_\phi(t)$ blows up at a finite number of discrete values of $t \in [0,\infty)$.  
The asymptotic blow-up direction for a negative real eigenvalue $\lambda_{\phi j}$ with multiplicity defined by an index set $\mathcal{J}$ is given by $\mathbf{u}_{\phi j}^\text{inf}$ which is a linear combination of the corresponding right eigenvectors of $G_\phi(t)$, $W_{\phi j}$ for $j \in \mathcal{J}$.  The specific coefficients of that linear combination are given by $V_{\phi j}^\top \big[\mathbf{u}_\text{A}-\mathbf{u}_\text{N}\big]$ for $j \in \mathcal{J}$ where $V_{\phi j}$ is the corresponding left eigenvector.

\begin{remark}
\label{rem:forlorn}
The implications of the above arguments are somewhat surprising. Specifically, the existence of negative eigenvalues implies that various social value orientations even in the altruistic regime, $(\theta_1, \theta_2) \in [0,\pi/2] \times  [0,\pi/2]$,
can cause the Nash equilibria to go unbounded. Intuitively this means that agents seeking to help each other in their optimization problems may cause unbounded actions.  Further, the structure of the argument indicates that for a given $\phi$, the number of problematic SVO values is finite and it is not clear where they will occur without a careful look at the above analysis, ie. agents may stumble onto them accidentally while experimenting with various cooperation values.  The problematic implications for designing cooperative schemes especially between humans and automata are obvious.  One could hope that these negative eigenvalues may never exist, but we will show with numerous counter-examples that they can exist and quite often will.
\end{remark}

\begin{remark}
By segmenting the eigenvalues of $\Lambda_\phi$ and $\Lambda_{\psi}$ into positive and negative eigenvalues and considering the corresponding eigenvectors, it would be straightforward to combine Props. \ref{prop:intersectellipses} and \ref{prop:intersectellipses4} with Prop. \ref{prop:asymp} to characterize the SVO-equilibria as being inside ellipses in the directions with positive eigenvalues and potentially blowing up in the directions with negative eigenvalues similarly to the way the state space of linear system can be divided into stable and unstable manifolds. For simplicity and easy of presentation, we do not explicitly spell out these constructions. 
\end{remark}



\subsubsection{Unbounded SVO Equilibria: Basic $2\times2$ Examples}
\label{sec:2Dcounter}.  
As mentioned in Remark \ref{rem:forlorn}, even simple two player scalar games can generate this unbounded behavior.  We give a simple example here visualizations to illustrate the blow up dynamic.   To this end we consider costs of the form 
\begin{align*}
    J_1 = \tfrac{1}{2}(u-\bar{u}_1)^\top 
    M_1
    (u-\bar{u}_1), \ \ 
J_2 = \tfrac{1}{2}(u-\bar{u}_2)^\top 
    M_2
    (u-\bar{u}_2)
    \end{align*}
with $\bar{u}_1 = (0.2,1.0)$ and $\bar{u}_2 = (1.0,0.2)$ and 
\begin{align}
M_1 
=
\begin{bmatrix}
1.2 & -1. \\ -1. & 1.
\end{bmatrix}, \ \ 
M_2
=
R(\gamma)
\begin{bmatrix}
1. & 1. \\ 1. & 1.2
\end{bmatrix}
R(\gamma)^\top 
\label{eq:rotexamples}
\end{align}
where $R(\gamma)$ is a $2 \times 2$ counter-clockwise rotation by angle 
for $\gamma \in [0,2\pi]$.  In Figs. \ref{fig:2x2bounded} and \ref{fig:2x2blowups} 
we show the set of SVO equilibria as the angle $\gamma$ varies with $\gamma = [\pi/8,0,-\pi/8]$ in Fig. \ref{fig:2x2bounded} and $\gamma = [7\pi/16,8\pi/16,9\pi/16]$ in Fig. \ref{fig:2x2blowups}.
Note how when the eigenvalues are not positive the SVO equilibria set blows up.  Note also the geometric nature of these counter examples and how the shape of the level sets interact with each other.  


\begin{figure}
    \centering
    \includegraphics[width=0.85\textwidth]{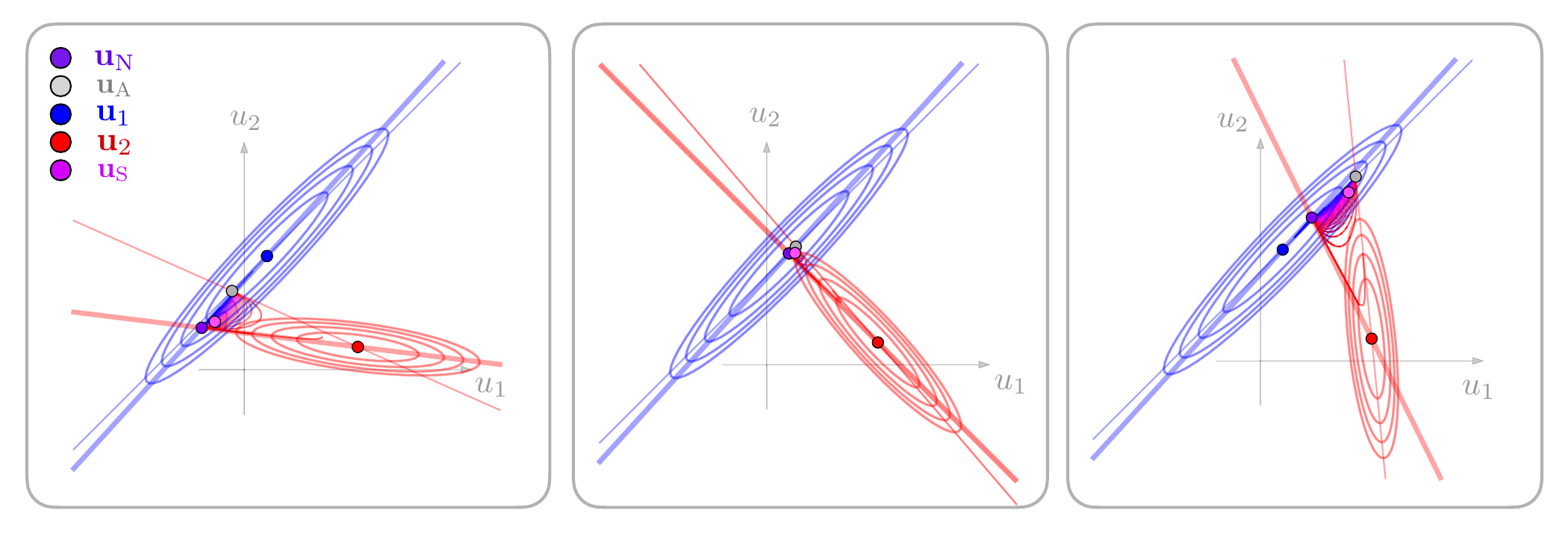}
    \caption{
    Set of SVO equilibria for $M_1$,$M_2$ in Eq. \eqref{eq:rotexamples} with (from left-to-right) $\gamma = \pi/8$, $\gamma = 0$, and $\gamma = -\pi/8$. In these examples both $\text{spec}(\Lambda_\phi) > 0$ and 
    $\text{spec}(\Lambda_\psi) > 0$ and the set of SVO-equilibria is quite localized. }
    \label{fig:2x2bounded}
\end{figure}

\begin{figure}
    \centering
    \includegraphics[width=0.85\textwidth]{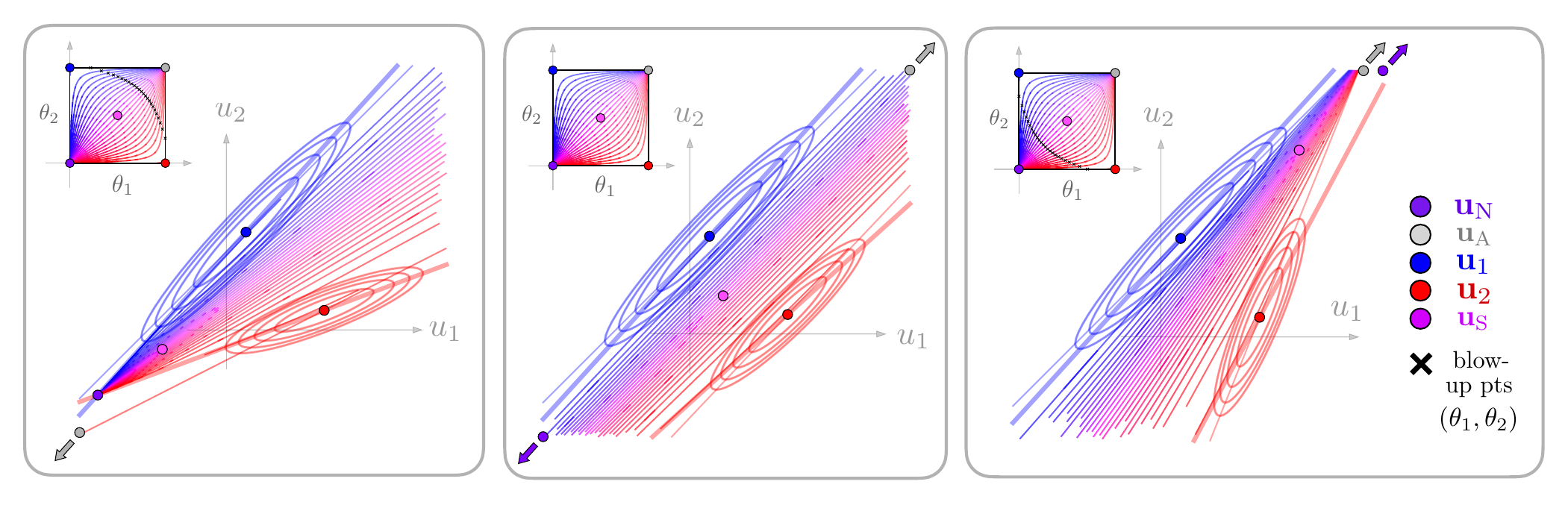}
\caption{Set of SVO equilibria (curves $\Gamma_\phi(t)$ for $M_1$,$M_2$ in Eq. \eqref{eq:rotexamples} with (from left-to-right) $\gamma = 7\pi/16$, $\gamma = 8\pi/16$, and $\gamma = 9\pi/16$. In these examples $\text{spec}(\Lambda_\phi) \ngtr 0$ and 
 and there are finite-blowup points for most values of $\phi$ as illustrated in the inset. Note how in the left and right figures both the Nash $\mathbf{u}_\text{N}$ and altruistic Nash are close to each other even tho they are at opposite ends of each curve $\Gamma_\phi(t)$. 
 Intuitively, the blow-up points are the values of $t$ for which the curves $\Gamma_\phi(t)$ cross infinity and come back around to the other side. }
    \label{fig:2x2blowups}
\end{figure}

    



\section{Applications: Open-loop LQ Games}

\label{sec:openloopLQ}

We now consider SVO-equilibria shifts in a two-player linear quadratic open-loop trajectory planning problem.  We first detail the linear time varying (LTV) open-loop problem and then apply it to a two-agent collision avoidance problem in 2D with single integrator dynamics for each agent.  

\subsection{Open-Loop LTV-Quadratic Games}

Consider the discrete time LTV system at finite time steps $k \in [0,\dots,K]$ with state $x_k \in \mathbb{R}^{n}$ at time $k$ and individual control inputs for two players at time $k$ for each player $u_{ik} \in \mathbb{R}^{m_i}$. Note that the state vector $x_k$ will often have components corresponding to the states of each individual player but we will store them in a joint state vector.  Define (discrete time) state transition matrix $F_k \in \mathbb{R}^{n \times n}$ and control input matrices $G_{ik} \in \mathbb{R}^{n \times m_i}$.
\begin{align*}
x_{k+1} = F_k x_k + G_{1k} u_{1k} + G_{2k} u_{2k}, \quad x_0 \ \text{given}
\end{align*}
for $k= 0,\dots,K-1$
with rolled out state trajectory $x = \big(x_k\big)_{k=1}^K \in \mathbb{R}^{nK}$. 
$u_i  = \big(u_{ik}\big)_{k=0}^{K-1} \in \mathbb{R}^{m_iK}$
where $d_1 = K\cdot m_1$ and $d_2 = K\cdot m_2$. 
Define matrices 
\begin{align*}
\mathbf{H} = 
\begin{bmatrix}
F_0 \\
F_1 F_0 \\
\vdots \\
\prod_{k=0}^K F_{k}
\end{bmatrix}, \quad 
\begin{matrix}
\mathbf{G}_i
= \textbf{blkdg}\Big(G_{i0},\dots,G_{i(K-1)}\Big), 
\end{matrix}
\end{align*}
and trajectory rollout matrices 
\begin{align*}
\mathbf{F} & = 
\begin{bmatrix}
I & 0 & 0 & \cdots & 0 \\
F_1 & I & 0 & \cdots & 0 \\
F_2F_1 & F_1 & I & \cdots & 0 \\
\vdots & \vdots & \vdots & \ddots & \vdots \\
\prod_{k=1}^{K-1} F_{k}
& 
\prod_{k=1}^{K-2} F_{k}
&
\prod_{k=1}^{K-3} F_{k}
& \cdots & I
\end{bmatrix}, 
\end{align*}
Note that the rolled out dynamcis can be compactly written as 
\begin{align*}
x = \mathbf{H}x_0 + \mathbf{F}\Big(\mathbf{G}_1u_1 + \mathbf{G}_2u_2\Big)
\end{align*}
Define running cost state matrices $Q_{ik} \in \mathbb{R}^{n \times n}$ and control cost matrices $R_{ik} \in \mathbb{R}^{m_i \times m_i}$ 
 for each player and combined block diagonal cost matrices 
\begin{align*}
&
\begin{matrix}
\mathbf{Q}_{i} = \textbf{blkdg}\Big[Q_{i1},\dots,Q_{iK} \Big]
\end{matrix},
\ 
\begin{matrix}
\mathbf{R}_{i} = \textbf{blkdg}\Big[R_{i0},\dots,R_{i(K\text{-}1)} \Big] 
\end{matrix}
\end{align*}
Define also the desired state trajectory for player $i$, $\bar{x}^i = (\bar{x}^i_k)_{k=1}^{K}$.  Note that the desired trajectory for each player contains the full state, not just their individual component of the state. 
Define the following costs for each player
\begin{align*}
J_i & = \sum_{k=1}^K \tfrac{1}{2}\big(x_k-\bar{x}_k^i\big)^\top Q_{ik}\big(x_k-\bar{x}_k^i\big) + \sum_{k=0}^{K-1}\tfrac{1}{2} u_{ik}^\top R_{ik}u_{ik} 
\\
& 
= \tfrac{1}{2}\big(x
-\bar{x}^i\big)^\top \mathbf{Q}_i \big(x-\bar{x}^i\big) + \tfrac{1}{2}u_i^\top \mathbf{R}_i u_i
\end{align*}
Using the rolled out form of the trajectories, we get cost matrices of the form. 
\begin{align*}
A_i & = \mathbf{R}_i + \mathbf{G}_i^\top \mathbf{F}^\top \mathbf{Q}_i \mathbf{F}\mathbf{G}_i, 
\quad 
a_i^\top = \big[\mathbf{H}x_0 - \bar{x}^i\big]^\top 
\mathbf{Q}_i \mathbf{F}\mathbf{G}_i \\
B_i & = \mathbf{G}_{-i}^\top \mathbf{F}^\top \mathbf{Q}_i \mathbf{F}\mathbf{G}_i, \quad \qquad \  
b_i^\top = \big[\mathbf{H}x_0-\bar{x}^i\big]^\top 
\mathbf{Q}_i \mathbf{F}\mathbf{G}_{-i} \\
D_i & = \mathbf{G}_{-i}^\top \mathbf{F}^\top \mathbf{Q}_i \mathbf{F}\mathbf{G}_{-i}
\end{align*}
Note the constant terms $\big[\mathbf{H}x_0- \bar{x}^i \big]^\top \mathbf{Q}_i\big[\mathbf{H}x_0- \bar{x}^i \big]$ have been removed from the costs. 
 Note specifically that in general, the cost terms $D_i$ are non-zero from the effect of the other players control action on the joint state.  
 
\subsection{Single-Integrator Trajectory Coordination}
\label{sec:lqcounter}

\subsubsection{Trajectory Parameters}
We now apply the above open-loop LQ-game framework to a simple two-agent trajectory coordination problem in 2D with single integrator dynamics for each agent.  Even in this simple example, we will see that for certian SVO values the SVO Nash-equilibria can go unbounded.  We will also derive some basic intuition for what the problematic blow-up points look like as inspiration for investigation in future work.  

We consider a time horizon of $K=50$ with time step $\Delta t = 0.04$. We consider a state vector $x_k \in \mathbb{R}^4$ where the first two elements are the position of player 1 and the last two elemenst are the position of player 2. 
Each player has full control over their velocity, ie. $u_{ik} \in \mathbb{R}^2$ and 
\begin{align*}
F_k = I_{4 \times 4}, \quad
G_{1k} = \Delta t \begin{bmatrix} I_{2 \times 2} \\ 0_{2 \times 2} \end{bmatrix}, \ \ 
G_{2k} = \Delta t \begin{bmatrix} 0_{2 \times 2} \\ I_{2 \times 2} \end{bmatrix}
\end{align*}
We take the following desired initial and terminal conditions $\bar{x}_0 = (-1.,0.,0.,-1.)$ and $\bar{x}_K = (1.,0.,0.,1.)$ and compute the desired trajectory (for both agents) by linearly interpolating between them. 
\begin{align*}
\bar{x}_k^i  = \big(\tfrac{k}{K}\big)\cdot \bar{x}_0 + \big(1-\tfrac{k}{K}\big)\cdot \bar{x}_K, \quad \text{for} \quad i=\{1,2\}
\end{align*}
These desired trajectories are illustrated in Fig. \ref{fig:lqbase}.

\begin{figure}
    \centering
    \includegraphics[width=0.65\textwidth]
    {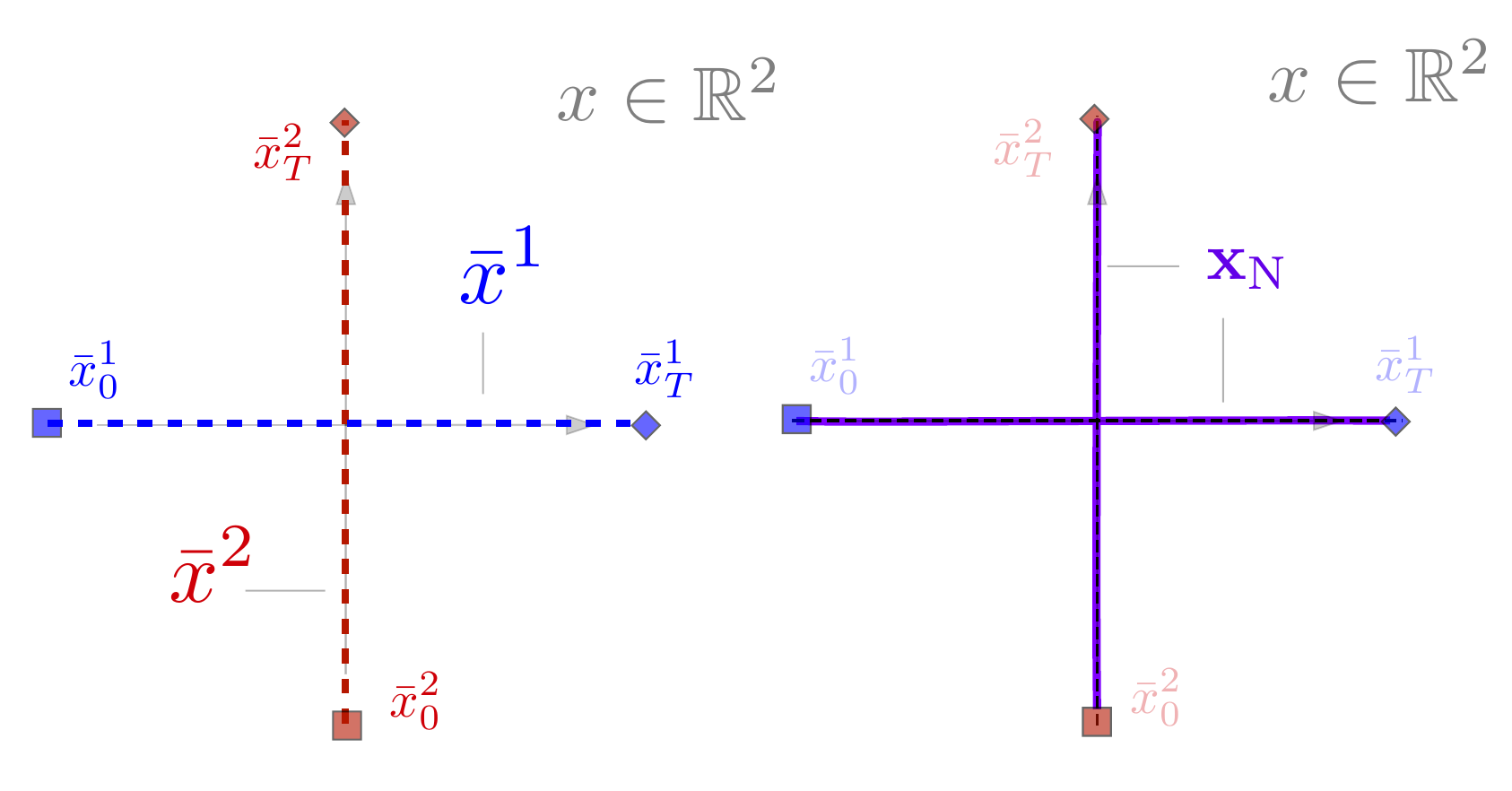}
    \caption{(Left) Desired trajectories $\bar{x}$ from initial conditions $\bar{x}_0 = (-1,0,0,-1)$ to desired terminal conditions $\bar{x}_K = (1,0,0,1)$. (Right) Original Nash equilibrium $\mathbf{x}_\text{N}$.  Note that since the tracking cost is relatively large, the Nash tracks almost exactly with the desired trajectories.}
    \label{fig:lqbase}    
\end{figure}

We set the cost for each agent to include a control effort cost, a tracking cost, and a proximity cost that penalizes agents being close to each other.
\begin{align*}
J_i
& = \sum_k
\big(x_{k}-\bar{x}^i_k\big)^\top Q^{\text{trk}}_{ik}
\big(x_{k}-\bar{x}^i_k\big)
+ 
u_{ik}^\top R_{ik}u_{ik} 
- 
x_k^\top Q_{ik}^{\text{prx}}x_k
\end{align*}
where for $i=\{1,2\}$, the general forms of the cost matrices is
$R_{ik} = r_{ik}I$,
and 
\begin{align*}
\begin{matrix}
Q^\text{trk}_{1k} =
\textbf{blkdg}\big(
 q^\text{trk}_{11}I, \  q^\text{trk}_{12}I
\big) \\
Q^\text{trk}_{2k} =
\textbf{blkdg}\big(
 q^\text{trk}_{21} I, \  q^\text{trk}_{22} I
\big)\Big.
\end{matrix},
\quad
Q^\text{prx}_{ik}
= 
q^\text{prx}_{ik}
\begin{bmatrix}
I & - I \\
-I & I
\end{bmatrix}
\end{align*}
 where the identity matrices are all $2 \times 2$.  The control scale parameters are given by $r_{ik} = 1.0$; and the tracking and proximity parameters are given by 
\begin{align*}
q^\text{trk}_{ii} =
\begin{cases}20 &; \ k < 45 \\ 30 &; \  45 \leq k \leq 50 
\end{cases}, \qquad 
q^\text{prx}_{ik} = 6.0
\end{align*}
where we penalize the later time steps more to ensure that that the terminal state is reached.  The other tracking scaling parameters 
$q^\text{trk}_{12} = q^\text{trk}_{21} = 0.1$ are regularization term so that the matrix $\mathbf{N}$ is invertible.   Note that the costs are symmetric for the two players in order to keep the problem simple but this could easily be changed.

We denote the original Nash equilibrium trajectories from $\mathbf{u}_\text{N}$, as 
$$\mathbf{x}_\text{N} = \mathbf{F}\big[\mathbf{G}_1 \ \mathbf{G}_2\big]\mathbf{u}_\text{N} + \mathbf{H}x_0$$
Similarly for a patricular SVO-equilibrium $\mathbf{u}_\theta$, we can denote the resulting state trajectory as 
$$\mathbf{x}_\theta = \mathbf{F}\big[\mathbf{G}_1 \ \mathbf{G}_2\big]\mathbf{u}_\theta + \mathbf{H}x_0 = 
\mathbf{F}\big[\mathbf{G}_1 \ \mathbf{G}_2\big]\Gamma_\phi(t) + \mathbf{H}x_0
$$
for $(\phi,t)= \Theta_\phi^{-1}(\theta)$.  
In order to visualize the asymptotic behavior in the trajectory space, 
let
\begin{align*}
\mathbf{x}_{\phi j}^\text{fin} = \mathbf{F}\big[\mathbf{G}_1 \ \mathbf{G}_2\big]\mathbf{u}_{\phi j}^\text{fin}, \qquad 
\mathbf{x}_{\phi j}^\text{inf} = \mathbf{F}\big[\mathbf{G}_1 \ \mathbf{G}_2\big]\mathbf{u}_{\phi j}^\text{inf}, \qquad 
\end{align*}
and 
we then have that for $\mathbf{u}_\theta = \Gamma_\phi(t)$ as $t \to |\lambda_{\phi j}|$ 
\begin{align*}
\mathbf{x}_\theta
& = 
\lim_{t \to |\lambda_{\phi j}| } 
\mathbf{F}\big[\mathbf{G}_1 \ \mathbf{G}_2\big]\Gamma_\phi(t) + \mathbf{H}x_0  \\
& = \mathbf{x}_\text{N} + \mathbf{x}_{\phi j}^\text{fin} + 
\lim_{t\to |\lambda_{\phi j}|}
\left(
1+\tfrac{\lambda_{\phi j}}{|\lambda_{\phi j}|} \right)^{-1}
\mathbf{x}_{\phi j}^{\text{inf}} 
\end{align*}
for any $\lambda_{\phi j} \in \text{spec}^-(\Lambda_\phi)$.

\subsubsection{Trajectory Discussion}


\begin{figure}
    \centering
    \includegraphics[width=0.5\textwidth]
    {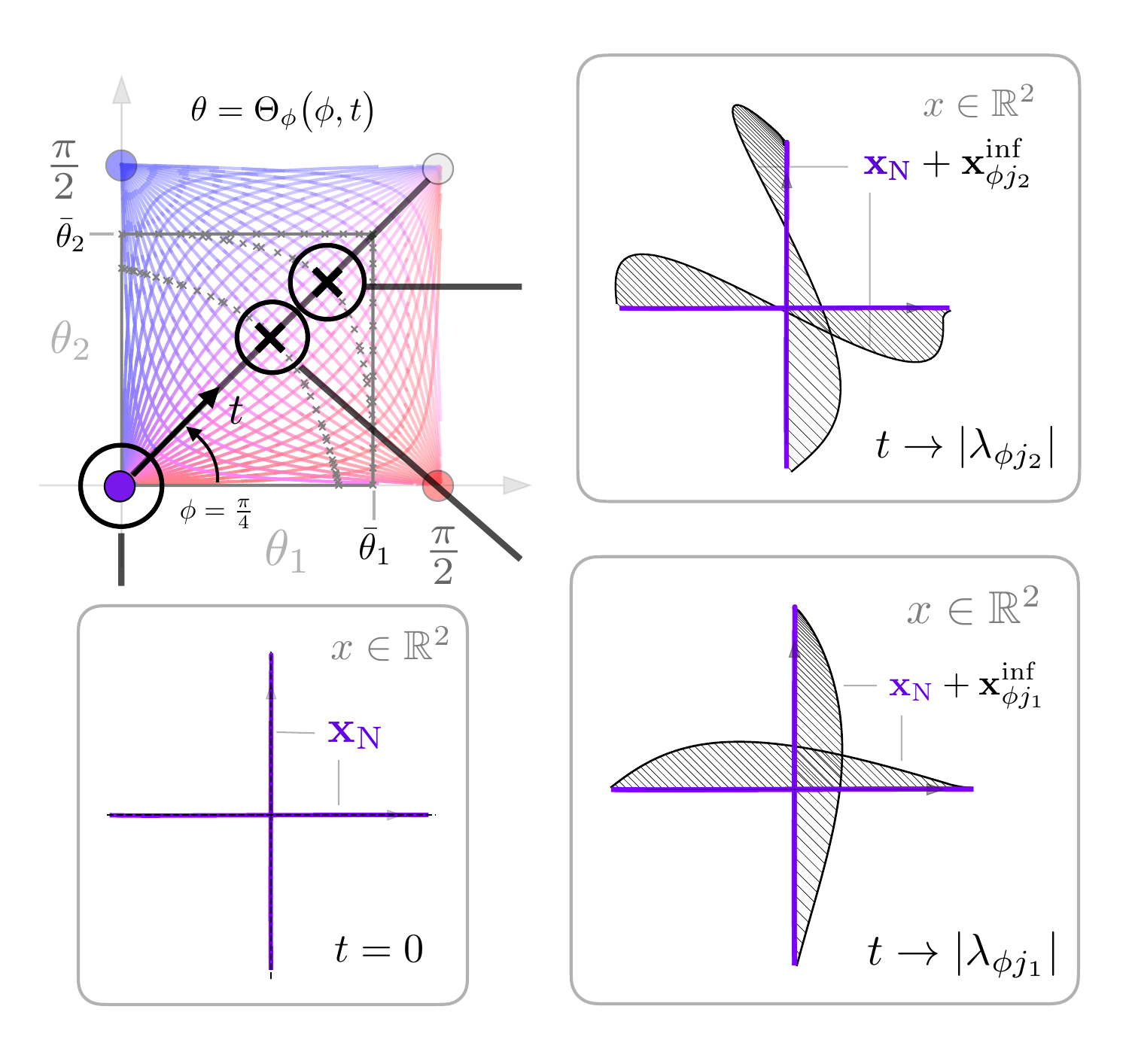}
\caption{
Problematic SVO values that cause unbounded SVO-equilibria. We focus on the curve $\phi = \pi/4$.   The insets show the asymptotic blow directions $\mathbf{u}_{\phi j_1}$ for $\lambda_{\phi j_1} = -1.1$ (bot-right) and $\mathbf{u}_{\phi j_2}$ for $\lambda_{\phi j_2} = -2.2$. (top-right). $\bar{\theta}_i$ give the maximum $\theta_i$ values for which the game is still well-posed. 
}
    \label{fig:lqexamples2}    
\end{figure}

\begin{figure}
    \centering
    \includegraphics[width=0.5\textwidth]
    {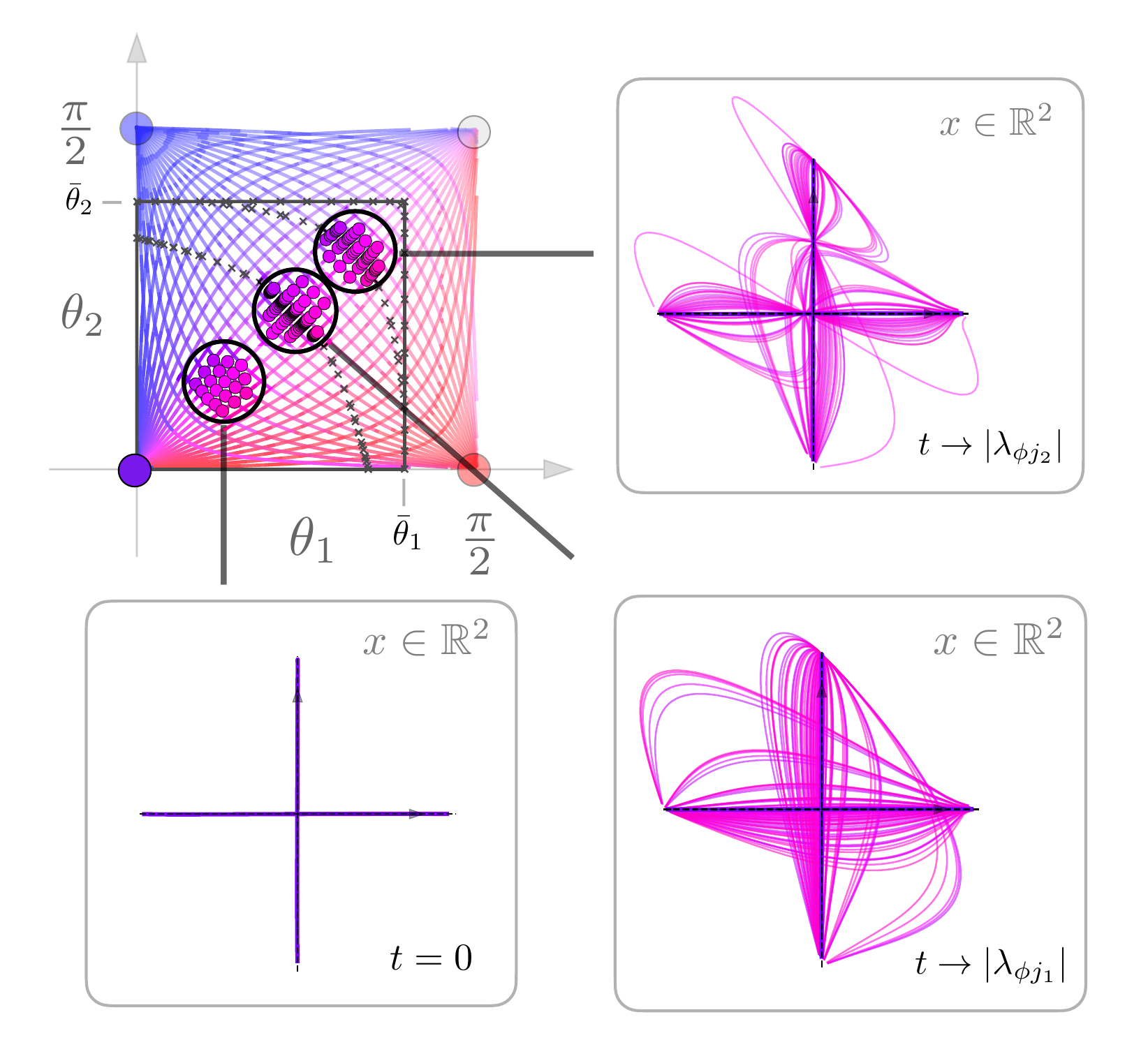}
\caption{
Sample SVO value around the problematic blow-up points on the curve $\phi = \pi/4$.  Note that for SVO values around $\theta = (\pi/8,\pi/8)$ (bot-left) the trajectories are still quite close to the original Nash trajectories.  For SVO-values sampled close to 
$t = \lambda_{\phi j_1} = -1.1$ (bot-right), we can see the trajectories get close to blowing up along the direction
$\mathbf{u}_{\phi j_1}$ (as shown in Fig.~\ref{fig:lqexamples2}) Similarly for SVO-values sampled close to 
$t = \lambda_{\phi j_1} = -2.2$ (top-right), we can see the trajectories get close to blowing up along the direction
and $\mathbf{u}_{\phi j_2}$ (as shown in Fig.~\ref{fig:lqexamples2}).  $\bar{\theta}_i$ give the maximum $\theta_i$ values for which the game is still well-posed.  }
    \label{fig:lqexamples1shexpansion}    
\end{figure}

Given the costs outlined above, the open-loop LQ game is well-posed for $\theta_i \leq 3\pi/8$.  For most values of $\phi$ there are two values of $t$ that cause the curve $\Gamma_\phi(t)$ to blow-up while the $\theta$ remains well-posed.  We focus on the curve where $\phi = \pi/4$ (ie. $\theta_1 = \theta_2$) since it is characteristic of the other curves and it contains the prosocial equilibrium $(\theta_1,\theta_2) = (\pi/4,\pi/4)$.  The blow-up values $\theta$ and the corresponding blow-up trajectory directions $\mathbf{x}_{\phi j}^\text{inf}$ are illustrated in Fig. \ref{fig:lqexamples2}.  In Fig. \ref{fig:lqexamples1shexpansion}, we show sampled values of $\theta$ close to the different blow up points.  Note that for $\theta = (\pi/8,\pi/8)$ the SVO-equilibrium trajectories are quite close to the Nash trajectories but then as they approach the two blow-up points shown the SVO-trajectories exhibit more erratic behavior in the blow-up directions as predicted.  In particular, empirically it seems like larger values $\theta$ cause higher frequency blow-up modes.  In future work, we hope to examine this phenomenon.  

\section{Discussion \& Conclusion}

\label{sec:conclusion}

In this paper, we have shown how Nash equilibria can shift under social value orientation in two player quadratic games (Section \ref{sec:svonash}) including open-loop linear time varying quadratic games.  
Under certain limited conditions presented in Section \ref{sec:contraction}, the new SVO-Nash equilibria remain bounded.  However, somewhat surprisingly as discussed in Section \ref{sec:blowups}, even in the cooperative regime, $\theta_i \in [0,\pi/2]$, different social value orientations can cause the Nash equilibrium to shift dramatically and even go unbounded at discrete cooperation levels.  
 Our empirical results indicate that these pathological unbounded solutions occur quite often in practice and 
could potentially cause problems for multi-agent coordination schemes unless care is taken.  

There are several interesting directions for further research. 
First, the stability of the SVO-equilibria in different regions of the cooperation space 
should be studied under various action-update schemes.  This characterization will be particularly important to understand how and when update schemes can converge to pathological equilibria. 
Second, initial empirical results seem to indicate that the pathological blowup points divide the cooperation space into discrete regions that each have topologically different types of solutions and that specifically for the case of LTV dynamic games, higher levels of cooperation can lead to behaviors with higher frequency oscillation modes.  Further work should be done to explore this phenomena and see how it can be leveraged for spatial and trajectory coordination.  
Finally, the expansions given could potentially be leveraged for learning problems where the goal is to estimate agents' costs and/or cooperation levels in two player competitive settings.  The asymptotic behavior detailed in Section \ref{sec:blowups} indicates that for any specific SVO-Nash equilibrium, the directions (eigenvectors) corresponding to the closest blow-up points in the cooperation space will have an out-sized effect on the nature of the SVO-equilibrium compared to the other eigenvectors.  This fact could be leveraged to do significant dimensionality reduction in estimation problems.

\section{Appendix}
\label{sec:appendices}

\subsection{Coordinate Transform Proofs of Prop. \ref{prop:coords}}
\label{app:coords12}

\begin{proof}
It suffices to define each inverse and note that they are valid on each transformation's co-domain $(\theta_1,\theta_2) \in (0,\pi/2) \times (0,\pi/2)$. 
\begin{align*}
\textbf{E1:} & \ \ (\phi,t) = 
\Theta_\phi^{-1}(\theta) = 
\left( \text{atan}
\left(
\frac{\tan \theta_2}{\tan \theta_1}
\right),
\frac{\tan \theta_1}{\cos\phi
}
= \frac{\tan \theta_2}{\sin\phi}\right) \\
\textbf{E2:} & \ \ (\psi,t) = 
\Theta_{\psi}^{-1}(\theta)
= 
\left(
\text{atan}
\left(
\frac{\cot \theta_2}{\tan \theta_1}
\right),
\frac{\tan \theta_1}{\cos\phi}
= \frac{\cot \theta_2}{\sin\phi}\right) \\
\textbf{E3:} & \ \ 
(\theta_1,t) 
= 
{\Theta_{1}}^{-1}(\theta) = \Big(\theta_1, 
\tan \theta_2 \Big) \\
\textbf{E4:} &  \ \ (t, \theta_2) = 
{\Theta_2}^{-1}(\theta) =
\Big(\tan \theta_1, \theta_2\Big)
\end{align*}
\end{proof}

\subsection{Expansions 3 \& 4}
\label{app:exp34}

In this section, we give the $\theta_1$-expansion ($\mathbf{E3}$) and the 
$\theta_2$-expansion ($\mathbf{E4}$). For each of these expansions we simply fix either $\theta_1$ or $\theta_2$ to generate curves and then sweeping $t$ to modify the other players cost.  

\begin{remark}
We give these expansions since the coordinate transformations \textbf{E3} and \textbf{E4} in Prop. \ref{prop:coords} are simpler and more straightforward. However in the analysis we will see that rather than isolating the affect of the SVO values to the matrices $G_\phi(t)$ and $G_\psi(t)$ as in Expansions \textbf{E1} and \textbf{E2}, the SVO values change every term in the expansion.  In addition the matrices $Z_{\theta_1}N_1^\top $ and 
$Z_{\theta_2}N_2^\top $ (detailed below) that have roles parallel to $\mathbf{M}H_\phi^{-1}\mathbf{N}^{-1}$ and $M_1H_\psi^{-1} M_2^{-1}$ are low rank and thus we can't make clean contraction arguments such as the ones detailed in Sec. \ref{sec:contraction}.  We include these expansions for the sake of being thorough and also since they provide more insight into parallel results for Stackelberg equilibria to be explored in future work. 
\end{remark}

In order to detail expansions \textbf{E3} and \textbf{E4}, we define 
\begin{align*}
    \mathbf{a}_{\theta_1} & = 
    \begin{bmatrix} \text{c} \theta_1 a_1+ \text{s} \theta_1 b_2  \\  a_2
    \end{bmatrix}, \ \
    \mathbf{M}_{\theta_1} = 
    \begin{bmatrix}
    \text{c}\theta_1 A_1 + \text{s} \theta_1 D_2 & B_2  \\
    \text{c}\theta_1 B_1 + \text{s}\theta_1 B_2^\top & 
    A_2
    \end{bmatrix} \\
    \mathbf{a}_{\theta_2} & = 
    \begin{bmatrix} 
    a_1 \\
    \text{c} \theta_2 a_2+ \text{s} \theta_2 b_1  
    \end{bmatrix}, \ \
    \mathbf{M}_{\theta_2}= 
    \begin{bmatrix}
    A_1 & 
    \text{c}\theta_2 B_2 + \text{s} \theta_2 B_1^\top \\
    B_1 & \text{c}\theta_2 A_2 +
    \text{s}\theta_2 D_1
    \end{bmatrix}
\end{align*}
and 
$
\mathbf{u}_{\theta_1}^\top  = -\textbf{a}_{\theta_1}^\top \mathbf{M}_{\theta_1}^{-1}
$, $
\mathbf{u}_{\theta_2}^\top = -\textbf{a}_{\theta_2}^\top \mathbf{M}_{\theta_2}^{-1} 
$. At several points in what follows we will want to use different portions of $\mathbf{M}_{\theta_1}^{-1}$and $\mathbf{M}_{\theta_2}^{-1}$.  Specifically we will define
\begin{align}
\mathbf{M}_{\theta_1}
^{-1}
& 
= 
\begin{bmatrix}
- \Big. Y_{\theta_1}^\top - \\
- Z_{\theta_1}^\top -  
\end{bmatrix}, \ \ 
\mathbf{M}_{\theta_2}
^{-1}
= 
\begin{bmatrix}
- \Big. Z_{\theta_2}^\top  -  \\
- Y_{\theta_2}^\top  - 
\end{bmatrix}
\label{eq:C3C4inv}
\end{align}
with $Y_{\theta_1}^\top,Z_{\theta_2}^\top \in \mathbb{R}^{d_1 \times d}$ and 
$Z_{\theta_1}^\top,Y_{\theta_2}^\top \in \mathbb{R}^{d_2 \times d}$
and also define $N_1^\top = \big[B_1 \ D_1^\top \big]$ and $N_2^\top = \big[D_2^\top \ B_2 \big]$.

We can now state the third expansion in the following proposition.
\begin{proposition}[\textbf{E3:} $\theta_1$-Expansion ]
\label{prop:theta1exp}
The SVO-equilibria $\mathbf{u}_\theta$ is given by 
\begin{align}
\mathbf{u}_\theta & = 
\mathbf{u}_{\theta_1} - G_{\theta_1}(t)
\Big[
Z_{\theta_1}b_1 +
Z_{\theta_1}N_1^\top 
\mathbf{u}_{\theta_1}
\Big] \notag 
\end{align}
where for $(\theta_1,t) \in (0,\pi/2) \times (0,\infty)$
\begin{align*}
G_{\theta_1}(t) 
& = 
\Big(\tfrac{1}{t}I + 
Z_{\theta_1}N_1^\top \Big)^{-1}
= 
\sum_i W_{i\theta_1}
\left(\tfrac{1}{t}+\lambda_{i\theta_1}\right)^{\text{-}1}
V_{i\theta_1}^\top 
\end{align*}
where $\big[N_1Z_{\theta_2}^\top \big]$ has eigenvalues $\Lambda_{\theta_1} = \{\lambda_{\theta_1 i}\}_{i=1}^d$, right-eigenvectors, 
$\{V_{\theta_1 i} \}_{i=1}^d$ and left-eigenvectors 
$\{W_{\theta_1 i}^\top \}_{i=1}^d$.
\end{proposition}

\begin{figure}
    \centering
\includegraphics[width=0.65\textwidth]
{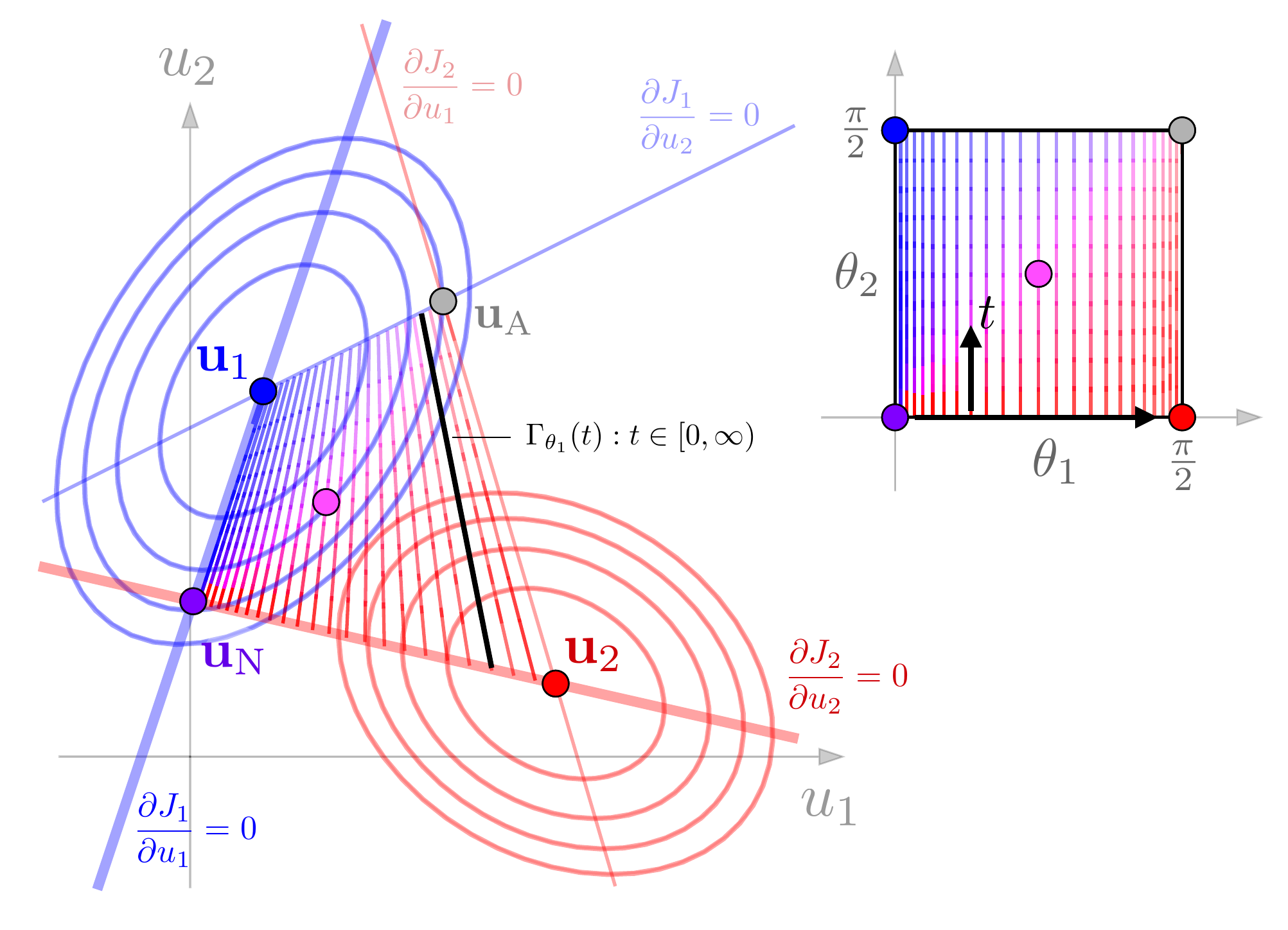}
\caption{(\textbf{E3}) $\theta_1$-Expansion: The curves $\Gamma_{\theta_1}(t):t \in [0,\infty)$ for $\theta_1 \in [0,\pi/2]$ are shown in color sweeping from 
    the point $\mathbf{u}_{\theta_1}$
    (Parameter values are given in App. \ref{app:values})
    }
    \label{fig:theta1exp}
\end{figure}
In the above proposition, we have used the following characterization of the spectrum of $G_{\theta_1}(t)$ as 
\begin{align*}
\text{spec}\Big(G_{\theta_1}(t)\Big)
= \Big\{ \big(\tfrac{1}{t} + \lambda_{i \theta_1}\big)^{-1} \ \Big| \ \lambda_{i\theta_1} \in \text{spec}\big(\Lambda_{\theta_1}\big)\Big\}
\label{eq:exp3spec}
\end{align*}
This expansion is illustrated in Fig. \ref{fig:theta1exp}.

\begin{figure}
    \centering
\includegraphics[width=0.65\textwidth]
{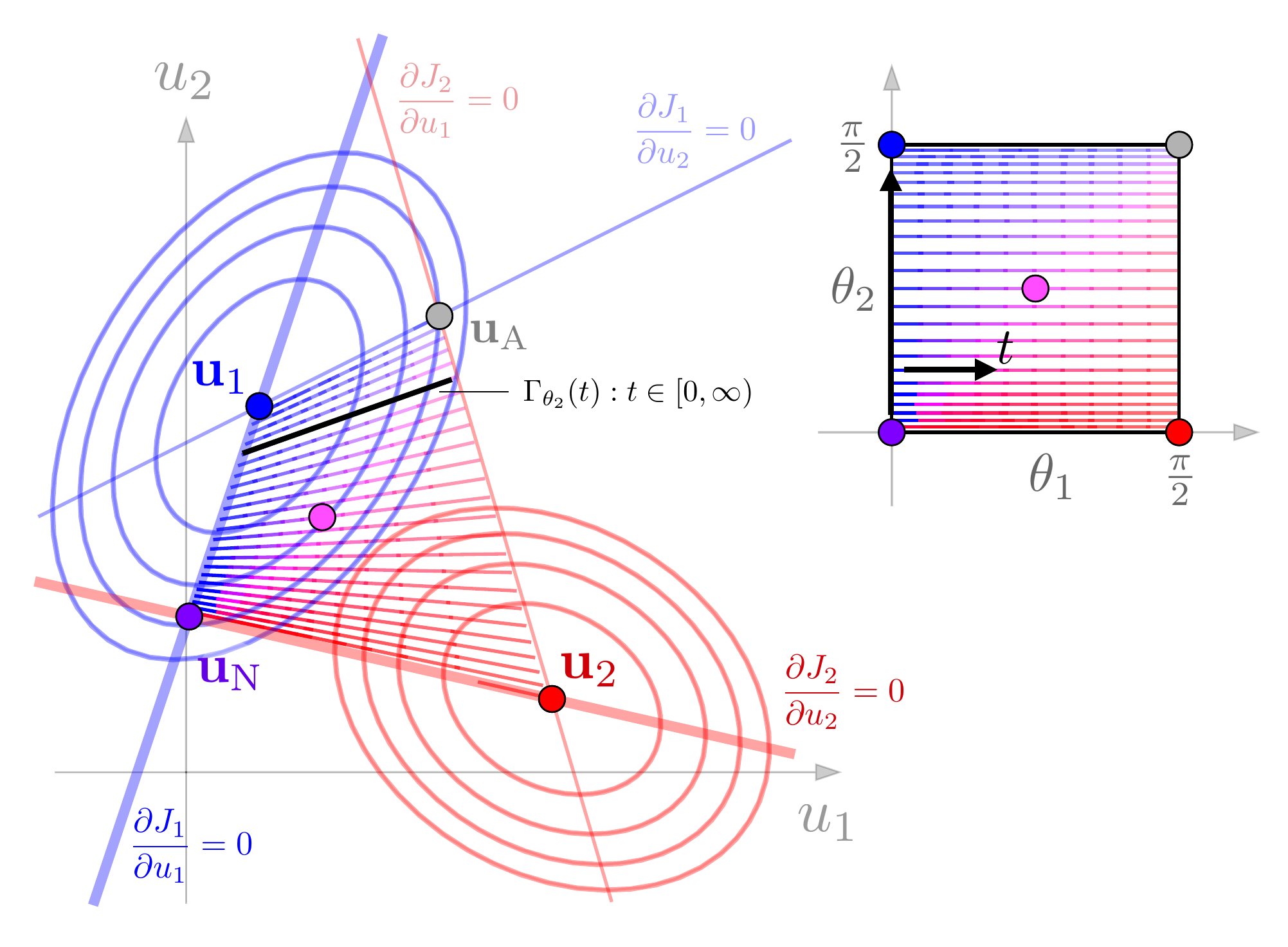}
    \caption{(\textbf{E4}) $\theta_2$-Expansion: The curves $\Gamma_{\theta_2}(t):t \in [0,\infty)$ for $\theta_1 \in [0,\pi/2]$ are shown in color sweeping from 
    the point $\mathbf{u}_{\theta_2}$
    (Parameter values are given in App. \ref{app:values})
    }
    \label{fig:theta2exp}
\end{figure}

Similarly, the fourth expansion can be given as follows 
\begin{proposition}[$\theta_2$-Expansion (\textbf{E4})]
\label{prop:theta2exp}
The SVO-equilibria $\mathbf{u}_\theta$ is given by 
\begin{align}
\mathbf{u}_\theta & = 
\mathbf{u}_{\theta_2} - G_{\theta_2}(t)
\Big[
Z_{\theta_2}b_2 +
Z_{\theta_2}N_2^\top 
\mathbf{u}_{\theta_2}
\Big] \notag 
\end{align}
where for $(\theta_1,t) \in (0,\pi/2) \times (0,\infty)$
\begin{align*}
G_{\theta_2}(t) 
& = 
\Big(\tfrac{1}{t}I + 
Z_{\theta_2}N_2^\top \Big)^{-1}
= 
\sum_i W_{i\theta_2}
\left(\tfrac{1}{t}+\lambda_{i\theta_2}\right)^{\text{-}1}
V_{i\theta_2}^\top 
\end{align*}
where $\big[N_2Z_{\theta_2}^\top\big]$ has eigenvalues $
\Lambda_{\theta_2} = 
\{\lambda_{\theta_2 i}\}_{i=1}^d$, right-eigenvectors, 
$\{V_{\theta_2 i} \}_{i=1}^d$ and left-eigenvectors 
$\{W_{\theta_2 i}^\top \}_{i=1}^d$.
\end{proposition}
In the above proposition, we have used the following characterization of the spectrum of $G_{\theta_2}(t)$ as 
\begin{align}
\text{spec}\Big(G_{\theta_2}(t)\Big)
= \Big\{ \big(\tfrac{1}{t} + \lambda_{i \theta_2}\big)^{-1} \ \Big| \ \lambda_{i\theta_2} \in \text{spec}\big(\Lambda_{\theta_2}\big)\Big\}
\end{align}

This expansion is illustrated in Fig. \ref{fig:theta2exp}.

\subsection{Proof of Expansions 1 \& 2}
\label{app:proofs}
The proofs of Expansions 1 \& 2 in Props. \ref{prop:nashexp} and \ref{prop:playerexp} follow parallel derivations and thus we can prove them together via the following lemma.  Other than the following lemma, the remainder of the equations are simple applications of the spectral mapping theorem.  

\begin{lemma}
\label{lem:expansions12}
Let 
\begin{align*}
u,a,b & \in \mathbb{R}^d, \ \ 
A,B \in \mathbb{R}^{d \times d}, \ \ 
H = \textbf{blkdg}\big(w_1 I, w_2 I\big) \\
\text{with} \quad 
u^\top & = - \big[a^\top + b^\top H \big]
\Big(A+BH\Big)^{-1}
\end{align*}
It follows that  
\begin{align}
u^\top 
& = u_a^\top + \big[u_b^\top  - u_a^\top \big] \big(AH^{-1}B^{-1}+I\big)^{-1} 
\end{align}
with $u_a^\top = -a^\top A^{-1}$ and $u_b^\top = -b^\top B^{-1}$.
\end{lemma}
\begin{proof}
\begin{align}
{u}^\top
& = 
-
\Big(a^\top + b^\top H\Big)
\Big(A + BH\Big)^{-1} \\
& = 
-a^\top 
A^{-1}\Bigg[I-B\Big(H^{-1} + A^{-1}B\Big)^{-1}A^{-1}\Bigg]
\\ & \qquad 
- b^\top B^{-1} 
\Big(AH^{-1}B^{-1} + I\Big)^{-1} \\
& = 
-a^\top 
A^{-1}\Bigg[I-\Big(AH^{-1}B^{-1} + I\Big)^{-1}\Bigg]
\\ & \qquad 
- b^\top B^{-1} 
\Big(AH^{-1}B^{-1} + I\Big)^{-1} \\
& = 
-a^\top 
A^{-1} - \Bigg[b^\top B^{-1}-a^\top A^{-1}\Bigg]\Big(AH^{-1}B^{-1} + I\Big)^{-1}
\label{eq:mainproof1}
\end{align}
The alternative expansions are given by writing
\begin{align*}
u^\top
& = -\Big(a^\top + b^\top H\Big)\Big(A+BH\Big)^{-1} \\
& = -\Big(a^\top H^{-1} + b^\top \Big)\Big(A H^{-1}+B\Big)^{-1} \\
& = 
- b^\top B^{-1} - \Big[a^\top A^{-1} - b^\top B^{-1}\Big]
\Big(BHA^{-1} + I \Big)^{-1}
\end{align*}
where the last step follows the same logic as 
in \eqref{eq:mainproof1} swapping the roles of 
\end{proof}

\subsubsection*{Expansions 1 \& 2}

\begin{proof}
The proof of expansion 1 is given by applying Lemma \ref{lem:expansions12} with $w_1 = t\cos \phi$ and $w_2 = t \sin \phi$
\begin{align*}
a = \mathbf{a}, \ \  b = \mathbf{b}, 
\ \ A = \mathbf{M}, \ \ B = \mathbf{N}.
\end{align*}
For Expansion 2, apply Lemma \ref{lem:expansions12} with $w_1 =  t \cos \psi$ and $w_2 =  t  \sin \psi$
\begin{align*}
a = c_1, \ \  b = c_2, 
\ A = M_1, \ B = M_2.
\end{align*}
\end{proof}

\subsection{Proof of Expansions 3 \& 4}
\label{app:proofs34}
 
\begin{lemma}
\label{lem:expansions12}
Let 
\begin{align*}
u,a,b & \in \mathbb{R}^d, \ \ 
A,B \in \mathbb{R}^{d \times d}, 
\ \ 
\\
\text{with} \quad 
u^\top & = - \big[a^\top + t b^\top \big]
\Big(A+B t\Big)^{-1}
`\end{align*}
It follows that  
\begin{align}
u^\top 
& = u_a^\top + \Big[b^\top  - u_a^\top B \Big]A^{-1} \Big(\tfrac{1}{t}+BA^{-1}\Big)^{-1} 
\end{align}
with $u_a^\top = -a^\top A^{-1}$. 
\end{lemma}

\begin{proof}
\begin{align*}
{u}^\top
& = 
-
\Big(a^\top + b^\top t\Big)
\Big(A + Bt\Big)^{-1} \\
& = 
-a^\top 
A^{-1}\Bigg[I-B\Big(\tfrac{1}{t}I + A^{-1}B\Big)^{-1}A^{-1}\Bigg]
\\ & \qquad 
- b^\top t
\Big(A+Bt\Big)^{-1} \\
& = 
-a^\top 
A^{-1}\Bigg[I-BA^{-1}\Big(\tfrac{1}{t}I + BA^{-1}\Big)^{-1}\Bigg]
\\ & \qquad 
- b^\top A^{-1}
\Big(\tfrac{1}{t}+BA^{-1}\Big)^{-1} \\
& = 
u_a^\top - \Big[b^\top A^{-1} + u_a^\top BA^{-1}\Big]
\Big(\tfrac{1}{t}+BA^{-1}\Big)^{-1} 
\end{align*}
\end{proof}

\subsubsection*{Expansions 3 \& 4}

\begin{proof}
For Expansion 1, apply Lemma \ref{lem:expansions12} with $w_1 = t\cos \phi$ and $w_2 = t \sin \phi$
\begin{align*}
a = \mathbf{a}, \ \  b = \mathbf{b}, 
\ \ A = \mathbf{M}, \ \ B = \mathbf{N}.
\end{align*}
For Expansion 2, apply Lemma \ref{lem:expansions12} with $w_1 =  t \cos \psi$ and $w_2 =  t  \sin \psi$
\begin{align*}
a = c_1, \ \  b = c_2, 
\ A = M_1, \ B = M_2.
\end{align*}
\end{proof}

\subsection{Example Parameters \& Constructions}
\label{app:values}

\subsubsection{2D \& 3D Examples}
\label{app:2Dn3Dexamples}
The 2D and 3D example parameters are given in the following form for visualization purposes. 
\begin{align*}
J_i = \tfrac{1}{2}\big(u-\mathbf{u}_i\big)^\top M_i \big(u - \mathbf{u}_i \big)
\end{align*}
The parameters in Figs. 
\ref{fig:nashillustration},
\ref{fig:nashexp}, 
\ref{fig:playerexp}, \ref{fig:theta1exp}, \ref{fig:theta2exp}, 
\ref{fig:balls}
\ref{fig:ellipsesNA12}, and 
\ref{fig:ellipsessamples}
are given by 
\begin{align*}
M_1 =
\begin{bmatrix}
3 & -1 \\ -1 & 2 
\end{bmatrix}, \ \ 
M_2 = 
\begin{bmatrix}
1.7  & 0.5 \\ 0.5 & 2.2    
\end{bmatrix}
\end{align*}
with $
\mathbf{u}_1 = (0.2, \ 1.0)$  and $\mathbf{u}_2 = 
(1.0, \ 0.2)$.   
For the 3D example in Fig. \ref{fig:3dexample},
\begin{align*}
M_1 =
\begin{bmatrix}
8.75 & 2.0 & 0.5 \\
2.0 & 3.75 & 1.0 \\
0.5 & 1.0 & 3.75 
\end{bmatrix}, \ \ 
M_2 = 
\begin{bmatrix}
1.8 & -1.6 &  0.0  \\
 -1.6 &  3.0 &  -2.8 \\
  0.0 &  -2.8 & 7.0 
\end{bmatrix}
\end{align*}
with 
$\mathbf{u}_1 = (-0.4, \ 0.2, \ 0.4)$ 
and
$\mathbf{u}_2 = (0.5, \ -0.4, \ 0.0)$
and $u^\top=[u_1^\top \ u_2^\top]$ with $u_1 \in \mathbb{R}^2$ and $u_2 \in \mathbb{R}$
(Parameters for Figs. \ref{fig:2x2bounded} and \ref{fig:2x2blowups} as well as the LQ examples in  Fig.
\ref{fig:lqexamples2}
\ref{fig:lqexamples1shexpansion} are given in the text.)



\bibliographystyle{plainnat}
\bibliography{refs}

\end{document}